\documentclass[leqno,11pt]{amsart}

\usepackage[top=1.5in,bottom=1.3in,left=1.3in,right=1.3in,marginparwidth=1.5cm]{geometry}

\usepackage[all]{xy}

\usepackage[utf8]{inputenc}

\usepackage{amsmath,amsthm,amssymb,amsfonts,mathrsfs,comment}

\usepackage[backref=page]{hyperref}
\renewcommand*{\backref}[1]{}
\renewcommand*{\backrefalt}[4]{%
	\ifcase #1 (Not cited.)%
	\or        (Cited on page~#2.)%
	\else      (Cited on pages~#2.)%
	\fi}

\hypersetup{
	colorlinks   =true,
	citecolor    =[rgb]{0.3,0,0.3},
	linkcolor    =blue,
	urlcolor     =[rgb]{0.3,0,0.3}
}

\usepackage[capitalise]{cleveref}

\numberwithin{equation}{section}
\usepackage{tikz}
\usetikzlibrary{cd}

\DeclareUnicodeCharacter{02BC}{'}

\newtheorem{lem}{Lemma}[section]
\newtheorem{prop}[lem]{Proposition}
\newtheorem{thm}[lem]{Theorem}

\newtheorem*{conjecture}{Conjecture}
 
\newtheorem*{thmA}{Theorem A}
\newtheorem*{thmB}{Theorem B}
\newtheorem*{thmC}{Theorem C}

\theoremstyle{definition}
\newtheorem{defn}[lem]{Definition}
\newtheorem{exmp}[lem]{Example}

\newtheorem{rem}[lem]{Remark}

\renewcommand{\P}{\mathbb{P}}
\newcommand{\Q}{\mathbb{Q}}
\newcommand{\Z}{\mathbb{Z}}
\newcommand{\R}{\mathbb{R}}
\newcommand{\C}{\mathbb{C}}
\renewcommand{\O}{\mathscr{O}}

\newcommand\td{{\rm{td}}}
\newcommand\ch{{\rm{ch}}}
\newcommand{\sO}{\mathcal{O}}
\newcommand{\cF}{\mathcal{F}}
\newcommand{\cX}{\mathcal{X}}
\newcommand{\fM}{\mathfrak{M}}
\newcommand{\cP}{\mathcal{P}}

\DeclareMathOperator{\divv}{div}
\DeclareMathOperator{\Deflt}{Def_{\operatorname{lt}}}
\DeclareMathOperator{\oH}{H}
\DeclareMathOperator{\Hilb}{Hilb}
\DeclareMathOperator{\Hom}{Hom}
\DeclareMathOperator{\Mon}{Mon^2_{\operatorname{lt}}}
\DeclareMathOperator{\Or}{O}
\DeclareMathOperator{\Pic}{Pic}
\DeclareMathOperator{\RR}{RR}
\DeclareMathOperator{\Aut}{Aut}
\DeclareMathOperator{\Hdg}{Hdg}

\usepackage{todonotes}
\makeatletter
\let\@wraptoccontribs\wraptoccontribs
\makeatother

\subjclass[2020]{Primary: 14J42}
\keywords{Singular symplectic varieties, lagrangian fibrations, SYZ conjecture, singular moduli spaces of sheaves on K3 surfaces}
\begin{document}
	
	\title{The SYZ conjecture for singular moduli spaces of sheaves on K3 surfaces}
	\author{Claudio Onorati}
	\address{\newline Università della Calabria, Department of Mathematics and Computer Science, via Pietro Bucci, 87036 Arcavacata di Rende (CS), Italy} \email{claudio.onorati@unical.it}
	
	\author{\'Angel David R\'ios Ortiz}
	\address{\newline Instituto de Matem\'aticas, Universidad Nacional Aut\'onoma de México, Unidad Cuernavaca, Av. Universidad s/n, Col. Lomas de Chamilpa, 62210 Cuernavaca, Morelos, México}
	\email{david.rios@im.unam.mx}
	
	\begin{abstract}
		In this paper we prove the SYZ conjecture and we compute the Huybrechts--Riemann--Roch polynomial for irreducible symplectic varieties that are locally trivial deformation equivalent to moduli spaces of sheaves on K3 surfaces.
		As an intermediate step in the argument, we generalise to the singular setting a result of Matsushita about moduli spaces of lagrangian fibrations of primitive symplectic varieties.
	\end{abstract}
	
	\maketitle
	
	\tableofcontents
	
	\section*{Introduction}
	The SYZ conjecture for irreducible holomorphic symplectic manifolds predicts that nef and isotropic line bundles are associated to lagrangian fibrations (cf.\ \cite[Conjecture~4.1]{Sawon:AbelianFibred}). Here isotropic means with respect to the Beauville--Bogomolov--Fujiki form. 
	The conjecture holds for all known irreducible holomorphic symplectic manifolds, we refer to \cite{DHMV2024} for an updated reference and for an important proof in the case of fourfolds satisfying a topological constraint (see also Section~\ref{section:SYZ}).
	
	Recently, considerable interest has arisen in the theory of \emph{singular symplectic varieties}. In this setting, one can formulate an analogous version of the SYZ conjecture.
	
	\begin{conjecture}[SYZ conjecture for primitive symplectic varieties]
		Let $X$ be a primitive symplectic variety and $L$ a line bundle on it. If $L$ is nef and isotropic, then there exists a lagrangian fibration $f\colon X\to B$ and a very ample line bundle $A$ on $B$ such that $L^{\otimes k}=f^*A$, for some $k\geq1$.
	\end{conjecture}
	
	Primitive symplectic varieties are compact K\"ahler spaces and they have a well-defined K\"ahler cone (see Section~\ref{section:symplectic}). Then a class of type $(1,1)$ is \emph{nef} if it belongs to the closure of the K\"ahler cone. 
	
	The first goal of this work is to establish the SYZ conjecture for a distinguished class of symplectic varieties, namely those that are locally trivial deformation equivalent to singular moduli spaces of sheaves on K3 surfaces. Let us fix the terminology. If $S$ is a projective K3 surface, we consider moduli spaces of sheaves $M_v(S,H)$, where $v=mw$ is a Mukai vector (here $m\geq1$ and $w$ is primitive) and $H$ is a $v$-general polarisation. We refer to Section~\ref{section:moduli spaces of sheaves on K3} for generalities about moduli spaces of sheaves on K3 surfaces. Under our assumptions, the spaces $M_v(S,H)$ are irreducible symplectic varieties (\cite{PR:SingularVarieties}). Any irreducible symplectic variety that is a locally trivial deformation of a variety of the form $M_v(S,H)$ as above, with $v=mw$ and $w^2=2(k-1)$, will be said to be of type $\operatorname{K3}^{[k]}_m$. 
	Symplectic varieties of type $\operatorname{K3}^{[k]}_m$ have dimension $2(k-1)\,m^2+2$.
	If $m=1$ the moduli space is smooth and of type $\operatorname{K3}^{[k]}$, which motivates our terminology.
	
	\begin{thmA}[Theorem~\ref{thm:SYZ for M(m,k)}]
		Let $X$ be a symplectic variety of type $\operatorname{K3}^{[k]}_m$. Assume that $m>1$ and that if $m=2$ then $k>2$. If $L$ is a primitive, nef and isotropic line bundle, then there exists a lagrangian fibration $f\colon X\to \P$ such that $L=f^*\sO_\P(1)$.
	\end{thmA}
	
	Notice that we prove a stronger version of the SYZ conjecture: we state that $L$ itself (and not a positive multiple of it) induces a lagrangian fibration and that the base of this fibration is a projective space. For smooth irreducible symplectic varieties, it is conjectured that the base of a lagrangian fibration is always a projective space; in the singular setting though there are examples of lagrangian fibrations over singular spaces (see for example \cite{Matsushita:OnBaseLagrFibr,LLX:CubicFivefold,Nanni:SYZNikulin}).
	
	The numerical hypotheses on $m$ and $k$ are not restrictive. In fact, if $m=1$, i.e.\ $X$ is smooth, then the result is known (\cite{Markman:LagrangianFibrations,Wieneck2016}); if $m=2$ and $k=2$, then $X$ admits a crepant resolution of singularities that is of type OG10 and the result follows from the SYZ conjecture for manifolds of type OG10 (\cite{MO}).
	
	To the best of our knowledge, this is the first class of singular symplectic varieties (not admitting a crepant resolution) for which the SYZ conjecture is proved. After this paper came out, two more results should be added to this list: in \cite{Buelli:MonSingKum} the author proves the SYZ conjecture for irreducible symplectic varieties locally trivial deformation equivalent to an Albanese fibre of a singular moduli spaces of sheaves on abelian surfaces; in \cite{Nanni:SYZNikulin} the author proves the SYZ conjecture for Nikulin orbifolds.
	
	Let us outline the main ideas behind the proof of Theorem~A. First of all, there exists a stratification of $X$ by singular loci (see \cite{Kaledin:FromThePoissonPOV}). In our situation (cf.\ \cite{KaledinLehnSorger}), the most singular stratum is an irreducible holomorphic symplectic manifold $Y$ of type $\operatorname{K3}^{[k]}$. The geometry of $X$ very closely resembles that of $Y$. For instance, if we denote by $i\colon Y\to X$ the closed embedding, then the pullback $i^*\colon\oH^2(X,\Z)\to\oH^2(Y,\Z)$ is $m$--times an isometry, and it induces an isomorphism between the respective monodromy groups (see \cite{OPR}). We use these results to reduce the classification of monodromy-orbits of primitive isotropic vectors in $\oH^2(X,\Z)$ to the same classification in $\oH^2(Y,\Z)$, which was performed by Markman (\cite[Section~2]{Markman:LagrangianFibrations}); the geometric outcome is Proposition~\ref{prop:quasi SYZ}.
	
	From here Theorem~A follows from a result about deformations of lagrangian fibrations, which is our second main result.
	
	First of all, for a primitive symplectic variety $\bar X$, we denote by $\Lambda$ an abstract lattice such that $\oH^2(\bar X,\Z)_{\operatorname{tf}}\cong\Lambda$. Then $\fM_\Lambda$ stands for the moduli space of marked pairs $(X,\eta)$, where $X$ is locally trivial deformation equivalent to $\bar X$ and $\eta\colon\oH^2(X,\Z)_{\operatorname{tf}}\to\Lambda$ is an isometry. We denote by $\fM^0_\Lambda$ a connected component of $\fM_\Lambda$. If $\ell\in\Lambda$ is an isotropic element, then $\fM_\ell^0\subset\fM_\Lambda^0$ is a connected component of the subspace of pairs $(X,\eta)$ such that $\eta^{-1}(\ell)$ is of type $(1,1)$ on $X$. Inside $\fM^0_\ell$ we consider the two spaces
	\[ \fM_\ell^{\operatorname{nef}}:=\left\{ (X,\eta)\in\fM_\ell^0\mid \eta^{-1}(\ell) \mbox{ is nef} \right\} \]
	and 
	\[ \fM_\ell^{\operatorname{lagr},B}:=\left\{ (X,\eta)\in\fM_\ell^0\mid \eta^{-1}(\ell)  \mbox{ defines a lagrangian fibration over } B \right\}. \]
	Here we say that $\eta^{-1}(\ell)$ defines a lagrangian fibration over $B$ if there is a lagrangian fibration $f\colon X\to B$ such that $\eta^{-1}(\ell)=f^*h_B$, where $h_B$ is the smallest very ample class on $B$. Notice that $B$ is always projective $\Q$-factorial of Picard rank $1$ (cf.\ Theorem~\ref{thm:KL instead of Sch}), so $h_B$ is well defined.
	We refer to Section~\ref{section: moduli spaces of lagrangian fibrations} for the precise definitions and constructions.
	
	\begin{thmB}[Theorem~\ref{thm:nef=lagr}]
		If $\fM_\ell^{\operatorname{lagr},B}\neq\emptyset$, then $\fM_\ell^{\operatorname{lagr},B}\subset\fM^0_\ell$ is open and dense. Moreover, in this case, we have an equality 
		\[ \fM_\ell^{\operatorname{lagr},B}=\fM_\ell^{\operatorname{nef}}. \]
	\end{thmB}
	
	This result generalises to the singular case a result of Matsushita (\cite[Lemma~3.4]{Matsushita:Isotropic2017}; see also Kamenova--Verbitsky, \cite[Theorem~3.4]{KamenovaVerbitsky:FamiliesOfLagrangianFibrations}). A practical corollary of Theorem~B is Theorem~\ref{thm:main result on lagr}, where we prove that two nef and isotropic line bundles that are connected by a suitable locally trivial parallel transport operator define lagrangian fibrations over $B$ as soon as one of them does. Notice that this does not a priori imply that the corresponding lagrangian fibrations are deformation equivalent (as lagrangian fibrations): in Theorem~\ref{thm:def as lagr fibr iff existence of suitable pto} we provide a proof of this fact assuming that the varieties are terminal and $\Q$-factorial.
	\medskip
	
	We conclude the introduction with another application of Theorem~B, where we compute the Huybrechts--Riemann--Roch polynomial for symplectic varieties of type $\operatorname{K3}^{[k]}_m$. Recall that the HRR polynomial is a locally trivial deformation invariant numerical polynomial that allows us to compute the Euler characteristic of a line bundle only in terms of its BBF square (see Section~\ref{section:HHRR}).
	
	\begin{thmC}[Theorem~\ref{thm: RR of K3n type}]
		Let $X$ be a symplectic variety of type $\operatorname{K3}^{[k]}_m$. Then the Huybrechts--Riemann--Roch polynomial of $X$ is of $\operatorname{K3}^{[n]}$-type, where $n = (k-1)m^2+1$, i.e.
		\[ \RR_X(t)=\binom{t/2 + n + 1}{n}. \]
	\end{thmC}
	
	Using the locally trivial deformation invariance of HRR polynomials, Theorem~C reduces to a computation in a special case, namely the one of a Beauville--Mukai system admitting an effective theta divisor. Performing some explicit computations concerning the theta divisor (see Section~\ref{section:theta}), which are of independent interest, we can deduce Theorem~C.

	%%%%%%%%%%%%%%%%%%%%%%%%%%%%%%%%%%%%%%%%%
	\subsection*{Structure of the paper}
	We collect in Section~\ref{section:preliminaries} the preliminary results needed in the rest of the paper. In particular, here the reader can find the main definitions and facts about the geometry of singular symplectic varieties. Section~\ref{section:LagrangianFibrations} contains basic facts about lagrangian fibrations and their infinitesimal behaviour. 
	In Section~\ref{section: moduli spaces of lagrangian fibrations} we prove Theorem~B and a useful corollary, i.e.\ Theorem~\ref{thm:main result on lagr}. We also prove Theorem~\ref{thm:def as lagr fibr iff existence of suitable pto}, giving sufficient conditions for two lagrangian fibrations to be deformation equivalent. Section~\ref{section:moduli spaces of sheaves on K3} contains some basic and useful results about moduli spaces of sheaves on K3 surfaces. Finally, in Section~\ref{section:HRR for MSK3} we prove Theorem~C and in Section~\ref{section:SYZ for M(m,k)} we prove Theorem~A.
	
	%%%%%%%%%%%%%%%%%%%%%%%%%%%%%%%%%%%%
	\subsection*{Acknowledgments}
	We are very grateful to Christian Lehn and Emanuele Macrì for useful conversations at several stages of this work. We also thank Mirko Mauri for answering our questions about \cite[Appendix~A]{EFGMS:Boundedness}, Andreas H\"oring for bringing \cite{DHP:MMP4Complex} to our attention and Giovanni Mongardi for a clarifying discussion. We also wish to thank Nikos Tsakanikas for several comments on the manuscript. Last but not least, we thank the anonymous referees for carefully reading our paper and for several questions and comments that cleaned some arguments, improved the exposition and sharpened some results.
	
	%%%%%%%%%%%%%%%%%%%%%%%%%%%%%%%%%%%
	\subsection*{Funding}
	Ángel David Ríos Ortiz was supported by the European Research Council (ERC) under the European Union’s Horizon 2020 research and innovation programme (ERC-2020-SyG-854361-HyperK). 
	
	Claudio Onorati was partially supported by the European Union - NextGenerationEU under the National Recovery and Resilience Plan (PNRR) - Mission 4 Education and research - Component 2 From research to business - Investment 1.1 Notice Prin 2022 - DD N. 104 del 2/2/2022, from title "Symplectic varieties: their interplay with Fano manifolds and derived categories", proposal code 2022PEKYBJ – CUP J53D23003840006.
	He is a member of INDAM-GNSAGA.

	\section{Preliminaries}\label{section:preliminaries}
	Throughout this paper, we work with primitive and irreducible symplectic varieties. This first section recalls the main definitions and results and establishes several generalisations.
	
	%%%%%%%%%%%%%%%%%%%%%%%%%%%%%%%
	\subsection{Primitive and irreducible symplectic varieties}\label{section:symplectic}
	If $X$ is a reduced normal complex analytic variety, as usual we denote by $\Omega_X^{[p]}=\left(\wedge^p\Omega_X\right)^{**}$ 
	the sheaf of reflexive $p$-forms. The varieties $X$ we are interested in have a pure Hodge structure of weight $2$ on $\oH^2(X,\C)$ and a real reflexive $1$-form $\omega\in\oH^{1,1}(X,\R)=\oH^1(X,\Omega^{[1]}_X)\cap\oH^2(X,\R)$ is called a \emph{K\"ahler form}. The set of K\"ahler classes in $\oH^{1,1}(X,\R)$ is a cone $\mathcal{K}_X$, called the K\"ahler cone (see \cite[Proposition~2.8]{BakkerLehn2022}). A reduced normal complex analytic variety with a K\"ahler form is called a \emph{K\"ahler space}. We refer to \cite{Varouchas} for generalities about K\"ahler spaces (see also \cite{BakkerLehn2022} for the special case of symplectic varieties).
	\begin{defn}\label{defn:symplectic}
		Let $X$ be a compact K\"ahler space.
		\begin{enumerate}
			\item A \textit{symplectic form} on $X$ is a closed reflexive 2-form $\sigma$ on $X$ which is non-degenerate at each point of $X_{\operatorname{reg}}$.
			\item If $\sigma$ is a symplectic form on $X$, the pair $(X,\sigma)$ is a \textit{symplectic variety} if for every (K\"ahler) resolution $f\colon\widetilde{X}\to X$ of the singularities of $X$, the holomorphic symplectic form $\sigma_{\operatorname{reg}}:=\sigma_{|X_{\operatorname{reg}}}$ extends to a holomorphic 2-form on $\widetilde{X}$. 
			\item The symplectic variety $(X,\sigma)$ is a \textit{primitive symplectic variety} if $\oH^{1}(X,\mathcal{O}_{X})=0$ and $\oH^{0}(X,\Omega_{X}^{[2]})=\mathbb{C}\sigma$.
			\item The symplectic variety $(X,\sigma)$ is an \textit{irreducible symplectic variety} if for every finite quasi-étale morphism $f\colon Y\to X$ the exterior algebra of reflexive forms on $Y$ is spanned by $f^{[*]}\sigma$.
		\end{enumerate}
	\end{defn}
	Recall that if $X$ and $Y$ are two normal analytic varieties, a \emph{finite quasi-étale morphism} $f\colon Y\to X$ is a finite morphism that is étale in codimension one. 
	
	Symplectic varieties always have rational singularities (\cite[Proposition~1.3]{Beauville2000}).
	We also wish to point out that irreducible symplectic varieties are always primitive, but the converse is not always true. 
	
	%%%%%%%%%%%%%%%%%%%%%%%%%%%%%%%
	\subsection{Locally trivial families}\label{section:families}
	Let us recall that a \textit{locally trivial family} is a proper morphism $f\colon\cX\to T$ of complex analytic varieties such that $T$ is connected and, for every point $x\in\cX$, there exist open neighbourhoods $V_x\subset \cX$ and $V_{f(x)}\subset T$, and an open subset $U_x\subset f^{-1}(f(x))$ such that there is an isomorphism over $T$
	$$V_x\cong U_x\times V_{f(x)}.$$

	\begin{defn}\label{def:lt}
		\begin{enumerate}
			\item A \emph{locally trivial family of primitive (resp.\ irreducible) symplectic varieties} is a locally trivial family whose fibres are all primitive (resp.\ irreducible) symplectic. 
			
			\item Two primitive symplectic varieties are said to be \emph{locally trivial deformation equivalent} if they are members of a locally trivial family of primitive symplectic varieties.
		\end{enumerate}
	\end{defn} 
	
	By \cite[Corollary~4.11]{BakkerLehn2022} a small locally trivial deformation of a primitive symplectic variety is again primitive symplectic. The same holds for irreducible symplectic varieties, provided certain hypotheses are imposed on the topology of the smooth locus, on the type of singularities allowed, or on the projectivity of the fibres (see \cite[Section~1.2]{OPR}). We recall here the following version, which is relevant for our purposes.
	
	\begin{prop}
		Let $X$ be a primitive symplectic variety.
		\begin{enumerate}
			\item (\cite[Lemma~5.20]{BakkerLehn2022}) If $X$ is $\Q$-factorial, then any small locally trivial deformation of $X$ is $\Q$-factorial.
			\item (\cite[Proposition~1.8]{OPR}) If $X$ is terminal and irreducible symplectic, then any small locally trivial deformation of $X$ is terminal and irreducible symplectic
		\end{enumerate}
	\end{prop}
	
	\begin{rem}
		By a result of Namikawa (see \cite[Main Theorem]{Namikawa}), any flat deformation of a $\Q$-factorial and terminal symplectic variety is locally trivial.
	\end{rem}

	%%%%%%%%%%%%%%%%%%%%%%%%%%%%%%%%
	\subsection{The BBF quadratic form}\label{section:BBF}
	Let $X$ be a primitive symplectic variety and let us consider the torsion free group $\oH^2(X,\Z)_{\operatorname{tf}}$. From now on, by abuse of notation, we will simply use the notation $\oH^2(X,\Z)$ for its torsion free part.
	\begin{prop}[\protect{\cite[Corollary~3.5, Section~5.1, Lemma~5.7]{BakkerLehn2022}}]\label{prop:BBF}
		If $X$ is a primitive symplectic variety, then $\oH^2(X,\C)$ is a pure weight-two Hodge structure. There exists a non-degenerate quadratic form $q_X$ on $\oH^2(X,\Z)$ of signature $(3,b_2(X)-3)$. Moreover, $q_X$ is invariant under locally trivial deformations of $X$.
	\end{prop}
	The quadratic form $q_X$ is called the \emph{Beauville--Bogomolov--Fujiki form} (BBF form, for short). With an abuse of notation, we will systematically confuse the quadratic form $q_X$ with the associated bilinear form. The pair $(\oH^2(X,\Z),q_X)$ is then a lattice, called the \emph{Beauville--Bogomolov--Fujiki lattice} (BBF lattice, for short). Notice that if $X_1$ and $X_2$ are two locally trivial deformation equivalent primitive symplectic varieties, then $(\oH^2(X_1,\Z),q_{X_1})$ and $(\oH^2(X_2,\Z),q_{X_2})$ are isometric as lattices.
	
	We also point out that if $X$ is an irreducible symplectic variety, then by \cite[Corollary~13.3]{GGK:Klt} it is simply connected and therefore the cohomology group $\oH^2(X,\Z)$ is already torsion free. Finally, we want to remark that the BBF form $q_X$ gives a natural isomorphism 
	\begin{equation}\label{eqn:iso}
		\oH^2(X,\Q)\cong\oH_2(X,\Q). 
	\end{equation}
	This can be used to naturally see the space of curves in $\oH_2(X,\Z)$ inside $\oH^2(X,\Q)$ (see also \cite[Definition~2.6]{LMP}). 
	
	If $D\in\oH^2(X,\Z)$ is the class of a divisor, then we denote by $D^\vee\in\oH_2(X,\Z)$ the class such that $q_X(D,\alpha)=D^\vee.\alpha$, for every $\alpha\in\oH^2(X,\Z)$. In particular, if $D$ is primitive and $\delta=\divv(D)$ is the divisibility of $D$, i.e.\ $\delta$ is the positive generator of the ideal $q_X(D,\oH^2(X,\Z))\subset\Z$, then $D^\vee=D/\delta$.
	
	Finally, let us recall the following fact.
	\begin{prop}[\protect{\cite[Theorem~1]{Schwald2020}}]\label{prop:Fujiki}
		Let $X$ be a primitive symplectic variety of dimension $2n$. There exists a positive constant $C\in\Q_{>0}$, depending only on the locally trivial deformation class of $X$, such that for every $\alpha\in\oH^2(X,\C)$ we have
		\[ \int_X \alpha^{2n}=C q_X(\alpha)^n. \]
	\end{prop}
	The constant $C$ is called the \emph{Fujiki constant}.
	\begin{rem}
		Notice that in \cite{Schwald2020} the author uses irreducible symplectic variety for what is now customarily called a primitive symplectic variety. Moreover, they assume that the variety is projective, but as remarked in \cite[Section~5.14]{BakkerLehn2022} the argument follows in the non-projective case as well. See also \cite[Proposition~5.15]{BakkerLehn2022} for a more general statement.
	\end{rem}
	
	%%%%%%%%%%%%%%%%%%%%%%%%%%%%%%
	\subsection{Infinitesimal Torelli theorems}\label{section:inf Torelli}
	Let $X$ be a primitive symplectic variety, $\Deflt(X)$ the Kuranishi space of locally trivial infinitesimal deformations of $X$ and 
	\[ \Omega(X)=\left\{ z\in\mathbb{P}\oH^2(X,\Z)\mid q_X(z)=0,\; q_X(z,\bar{z})>0 \right\} \]
	the period domain.
	
	\begin{prop}[\protect{\cite[Lemma 4.6, Theorem 4.7, Proposition 5.5]{BakkerLehn2022}}]\label{prop:local Torelli}
		With notation as above, we have:
		\begin{enumerate}
			\item the Kuranishi space $\Deflt(X)$ is universal and smooth with tangent space isomorphic to $\oH^1(T_X)\cong\oH^1(X,\Omega_X^{[1]})$;
			\item (Local Torelli Theorem) if $\mathcal{X}\to\Deflt(X)$ is the universal family, then the period map
			\[ \Deflt(X)\longrightarrow\Omega(X),\quad t\mapsto [\oH^{2,0}(\mathcal{X}_t)] \]
			is a local isomorphism.
		\end{enumerate}
	\end{prop}
	
	Suppose now that $L\in\Pic(X)$ is a line bundle and $\Deflt(X,L)$ is the Kuranishi space of infinitesimal deformations of the pair $(X,L)$.
	
	\begin{prop}[\protect{\cite[Lemma 4.13, Corollary 5.9]{BakkerLehn2022}}]
		With notation as above, we have:
		\begin{enumerate}
			\item the Kuranishi space $\Deflt(X,L)$ is universal and smooth, the forgetful morphism $\Deflt(X,L)\to\Deflt(X)$ is a closed embedding, and the tangent space of $\Deflt(X,L)$ is isomorphic to
			\[ \operatorname{ker}\left( \oH^1(T_X)\stackrel{\cup c_1(L)}{\longrightarrow}\oH^2(\sO_X)\right); \]
			\item the image of $\Deflt(X,L)$ via the period map is identified with the space 
			\[ \Omega(X,L)=\mathbb{P}(c_1(L)^{\perp_{q_X}})\cap\Omega(X). \]
		\end{enumerate}
	\end{prop}
	
	As a consequence, we get that up to shrinking $\Deflt(X)$ if necessary, there exists a line bundle $\mathcal{L}$ on $\mathcal{X}_L:=\mathcal{X}\times_{\Deflt(X)}\Deflt(X,L)$ such that $(\mathcal{X}_L,\mathcal{L})$
	is a universal family of $\Deflt(X,L)$.
	
	%%%%%%%%%%%%%%%%%%%%%%%%%%%%%%
	\subsection{Locally trivial monodromy group}\label{section:Mon}
	Let $\pi\colon\mathcal{X}\to B$ be a locally trivial family of primitive symplectic varieties. For any $b\in B$, the lattices $\oH^2(\mathcal{X}_b,\Z)$ fit together to form a local system $R^2\pi_*\Z$, which comes with the Gauss--Manin connection. Therefore if $\gamma\colon [0,1]\to B$ is any path starting from a point $b_1$ and ending at a point $b_2$, then there is an isometry 
	\[ \mathsf{P}_\gamma\colon\oH^2(\mathcal{X}_{b_1},\Z)\longrightarrow\oH^2(\mathcal{X}_{b_2},\Z) \]
	obtained by parallel transport.
	
	\begin{defn}
		Let $X$, $X_1$ and $X_2$ be primitive symplectic varieties that are locally trivial deformation equivalent.
		\begin{enumerate}
			\item An isometry $g\colon \oH^2(X_1,\Z)\to\oH^2(X_2,\Z)$ is a \emph{locally trivial parallel transport operator} if there exist a locally trivial family $\pi\colon\mathcal{X}\to B$ and a path $\gamma\colon [0,1]\to B$ with $\mathcal{X}_{\gamma(0)}=X_1$ and $\mathcal{X}_{\gamma(1)}=X_2$, and such that $g=\mathsf{P}_\gamma$.
			\item An isometry $g\in\Or(\oH^2(X,\Z))$ is a \emph{locally trivial monodromy operator} if it is a locally trivial parallel transport operator from $X$ to itself.
			\item The \emph{monodromy group} $\Mon(X)$ is the group of locally trivial monodromy operators on $X$.
		\end{enumerate}
	\end{defn}
	
	%%%%%%%%%%%%%%%%%%%%%%%%%%%%%%%
	\subsection{The marked moduli space}\label{section:marked moduli space}
	Let $\Lambda$ be a lattice of signature $(3,n)$. The \emph{period domain} of $\Lambda$ is the domain
	\begin{equation}\label{eqn:period domain} 
		\Omega_\Lambda=\{ p\in\P(\Lambda\otimes_\Z\C)\mid (p,p)=0,\; (p,\bar{p})>0\} 
	\end{equation}
	
	Now let $\bar X$ be a fixed primitive symplectic variety, and denote by $\fM_\Lambda$ the moduli space of \emph{marked} primitive symplectic varieties locally trivial deformation of $\bar X$, i.e.\ $(X,\eta)\in\fM_\Lambda$ if and only if $X$ is locally trivial deformation equivalent to $\bar X$ and $\eta\colon\oH^2((X,\Z)_{\operatorname{tf}}\to\Lambda$ is an isometry. The space $\fM_\Lambda$ exists as a non-Hausdorff complex manifold of dimension $\operatorname{rk}(\Lambda)-2$ and it is constructed by gluing together the Kuranishi spaces $\Deflt(X)$ using the markings.
	
	The \emph{period map} is
	\[ \cP\colon\fM_\Lambda\longrightarrow\Omega_\Lambda,\qquad (X',\eta')\mapsto [\eta'(\sigma_X)]. \]
	By the Local Torelli Theorem (see Proposition~\ref{prop:local Torelli}), $\cP$ is a local isomorphism.
	
	In the following we denote by $\overline{\fM}_\Lambda$ the Hausdorff reduction of $\fM_\Lambda$ (see for example \cite[Corollary~4.10]{Huybrechts:Asterisques}) and by $\overline{\cP}$ the induced period map.
	
	\begin{prop}[Global Torelli Theorem, \protect{\cite[Theorem~1.1]{BakkerLehn2022}}]\label{thm:global Torelli}
		Assume that $\operatorname{rk}(\Lambda)\geq5$ and let $\fM^0_\Lambda$ be a connected component of $\fM_\Lambda$. Then 
		\begin{enumerate}
			\item $\cP\colon\fM^0_\Lambda\to\Omega_L$ is bijective over Mumford--Tate general points;
			\item $\overline{\cP}|_{\overline{\fM^0}_\Lambda}$ is an isomorphism onto its image, which is contained in the complement of countably many maximal Picard rank periods;
			\item if there exists $(X',\eta')\in\fM^0_\Lambda$ such that $X'$ is $\Q$-factorial and terminal, then $\overline{\cP}$ is surjective.
		\end{enumerate}
	\end{prop}
	
	Finally, let $\ell\in\Lambda$ be a class and put
	\[ \Omega_\ell=\{p\in\Omega_\Lambda\mid (p,\ell)=0\}\qquad\mbox{ and }\qquad \fM_\ell=\cP^{-1}(\Omega_\ell). \]
	By definition we have that $(X',\eta')\in\fM_\ell$ if $\eta'^{-1}(\ell)$ is of type $(1,1)$. 
	Notice that if $(X',\eta')\in\fM_\ell$ then an infinitesimal neighbourhood of $(X',\eta')$ is isomorphic to the Kuranishi space $\Deflt(X',L')$, where $L'$ is a line bundle on $X'$ such that $c_1(L')=\eta'^{-1}(\ell)$.
	
	%%%%%%%%%%%%%%%%%%%%%%%%%%%%%%%
	\subsection{Orientations}\label{section:orientation}
	
	Let $\Lambda$ be a lattice of signature $(3,n)$.
	The cone $\widetilde{\mathcal{C}}_{\Lambda}=\{ x\in\Lambda\otimes_\Z\R\mid (x,x)>0\}$ is connected and $\oH^2(\widetilde{\mathcal{C}}_{\Lambda},\Z)=\Z$ (\cite[Lemma 4.1]{Markman:Survey}). Any of the two generators of $\oH^2(\widetilde{\mathcal{C}}_{\Lambda},\Z)$ is an \emph{orientation} of $\widetilde{\mathcal{C}}_{\Lambda}$ (and corresponds to an orientation of a real positive 3-space of $\Lambda_\R$).
	
	Now let $\ell\in\Lambda$ be a class with $\ell^2=0$. As in the previous section, let us put $\Omega_\ell=\{p\in\Omega_\Lambda\mid (p,\ell)=0\}$ and notice that it has two connected components. Following \cite[Section 4.3]{Markman:LagrangianFibrations}, the choice of an orientation on $\widetilde{\mathcal{C}}_{\Lambda}$ determines one of the two connected components of $\Omega_\ell$. Let us recall how.
	
	First of all, if $p\in\Omega_\Lambda$ is a period, then $p$ determines a weight $2$ Hodge structure on $\Lambda$. If we denote by $\Lambda^{1,1}_\R(p)=\{x\in\Lambda\otimes_{\Z}\R\mid (x,p)=0\}$ the real part of type $(1,1)$, then the cone $\widetilde{\mathcal{C}}_p=\{x\in\Lambda^{1,1}_\R(p)\mid (x,x)>0\}$ has two connected components. As explained in \cite[Section 4.3]{Markman:LagrangianFibrations}, the choice of an orientation of $\widetilde{\mathcal{C}}_\Lambda$ uniquely determines the choice of a connected component of $\widetilde{\mathcal{C}}_p$. Now, by definition $\ell\in\Lambda^{1,1}_\R(p)$, and in fact it belongs to the closure of only one connected component of $\widetilde{\mathcal{C}}_p$, which by the discussion above corresponds to an orientation of $\widetilde{\mathcal{C}}_\Lambda$. 
	Therefore, once an orientation on $\widetilde{\mathcal{C}}_\Lambda$ is fixed, a connected component of $\Omega_\ell$ is chosen by requiring that $\ell$ belongs to the determined connected component of $\widetilde{\mathcal{C}}_p$.
	
	Now let $X$ be a primitive symplectic variety. Since $\oH^{1,1}(X,\R)$ is of signature $(1,b_2(X)-3)$, the cone of positive classes $\{x\in\oH^{1,1}(X,\R)\mid (x,x)>0\}$ has two connected components.  The \emph{positive cone} of $X$ is then the distinguished connected component $\mathcal{C}_X$, of the cone of positive classes, containing the K\"ahler cone (cf.\ \cite[Section 2.3]{BakkerLehn2022}). If $\eta$ is a marking of $X$ and we put $p=\mathcal{P}(X,\eta)$, then $\eta(\mathcal{C}_X)$ is a distinguished connected component of $\widetilde{\mathcal{C}}_p$, and hence it determines an orientation of $\widetilde{\mathcal{C}}_\Lambda$. 
	
	\begin{rem}
		Let $X$ and $Y$ be two primitive symplectic varieties. An isometry 
		\[ g\colon \oH^2(X,\Q)\longrightarrow \oH^2(Y,\Q) \]
		comes in two flavours: either it is orientation--preserving or it is orientation reversing. Geometrically, this can be interpreted by saying that $g$ is orientation--preserving if it sends the positive cone of $X$ onto the positive cone of $Y$.
		
		In particular, locally trivial parallel transport operators are orientation--preserving.
	\end{rem}
	
	If $(X,\eta)$ varies in a connected component of the corresponding moduli space, then the corresponding orientation remains fixed: a connected component $\mathfrak{M}_\Lambda^0$ of $\mathfrak{M}_\Lambda$ determines an orientation of $\widetilde{\mathcal{M}}_\Lambda$ (cf.\ \cite[Section 4]{Markman:Survey}). 
	
	By the discussion at the beginning of this section, the choice of a connected component $\mathfrak{M}_\Lambda^0$ then determines a connected component $\Omega_\ell^+$ of $\Omega_\ell$. If $\mathcal{P}_0$ denotes the restriction of the period map $\mathcal{P}$ to $\mathfrak{M}_\Lambda^0$, then we define 
	\begin{equation}\label{eqn: M ell 0} 
		\mathfrak{M}_\ell^0=\mathcal{P}_0^{-1}(\Omega_\ell^+).
	\end{equation}

	%%%%%%%%%%%%%%%%%%%%%%%%%%%%%%%
	\subsection{Prime exceptional divisors}\label{section:pex}
	The following definition is \cite[Definition~5.1]{Markman:Survey} for smooth symplectic varieties, which can be extended to singular ones without any change.
	\begin{defn}\label{defn:pex}
		Let $X$ be a primitive symplectic variety and $D\subset X$ an irreducible and reduced effective $\Q$-Cartier divisor. Then $D$ is \emph{prime exceptional} if $q_X(D)<0$.
	\end{defn}
	Prime exceptional divisors are uniruled and, if we denote by $\ell\in\oH_2(X,\Z)$ the class of a general curve in the ruling, then $D^\vee$ and $\ell$ are proportional by a rational constant (see \cite[Theorem~1.2.(1)]{LMP}). Vice versa, assume that $X$ is projective and let $\ell\in\oH_2(X,\Z)$ be the class of a rational curve ruling a divisor $D$; if $\ell$ is smooth and $D$ is Cartier, then $D$ is prime exceptional and $D^\vee$ and $\ell$ are proportional by a rational constant (see \cite[Lemma~3.13, Theorem~1.1]{LMP}). Finally, let us notice that prime exceptional divisors deform over a non-empty open subset of their Hodge locus (see \cite[Theorem~1.2.(2)]{LMP}).
	
	%%%%%%%%%%%%%%%%%%%%%%%%%%%%%%
	\subsection{The Huybrechts--Riemann--Roch polynomial}\label{section:HHRR}
	Let us recall the following result.
	\begin{thm}\label{RRpolynomial}
		\cite[Corollary~5.16]{BakkerLehn2022} Let $X$ be a primitive symplectic variety. There exists a unique polynomial $\RR_X(t)\in \Q[t]$ such that for any line bundle $L$ on $X$, it holds $\RR(q(c_1(L))) = \chi(L)$. Moreover, $\RR_X = \RR_{X'}$ for every locally trivial deformation $X'$ of $X$. 
	\end{thm}
	
	\begin{defn}
		Let $X$ be a primitive symplectic variety. Define the \emph{Huybrechts--Riemann--Roch polynomial} of $X$ to be the polynomial $\RR_X(t)$ in \cref{RRpolynomial}.
	\end{defn}
	
	\begin{rem}
		When $X$ is smooth, by \cite[Corollary~23.17]{GHJ2003} for any $\alpha\in \oH^{4k}(X,\Q)$ that is of type $(2k,2k)$ for all small deformations of $X$, there exists a constant $C(\alpha)\in\Q$ such that
		\begin{equation}\label{eq:todd}
			\int_X \alpha \smile \beta^{2n-2k} = C(\alpha)\cdot q_X(\beta)^{n-k}
		\end{equation}
		for all $\beta\in \oH^2(X,\Q)$. Combining this with the Riemann--Roch--Hirzebruch formula we get
		
		\begin{equation}\label{eq:chi}
			\chi(X,L) =  \sum_{i=0}^n\frac{1}{(2i)!}\int_X Td_{2n-2i}(X)\smile c_1(L)^{2i} = \sum_{i=0}^n \frac{a_i}{(2i)!} \cdot q_X(L)^{i}
		\end{equation}
		where $a_i := C(Td_{2n-2i}(X))$. Hence in the smooth case the Huybrechts--Riemann--Roch polynomial is $\RR_X(t) = \sum_{i=0}^n \frac{a_i}{(2i)!}t^i$.
	\end{rem}
	
	In the case of irreducible symplectic varieties with orbifold singularities the Huybrechts--Riemann--Roch polynomial can be computed as in \cite[Section 3]{BeckmanSong2022}.
	
	\begin{exmp}\label{exmp: RR for known HK}
		Huybrechts-Riemann-Roch polynomials for the deformation classes constructed by Beauville were computed in \cite{EGM01} and \cite{Nieper03}. Explicitly, if $X$ is an irreducible holomorphic symplectic manifold of type $\operatorname{K3}^{[n]}$, then the Huybrechts--Riemann--Roch polynomial is given by
		\[
		\RR_X(t) = \binom{t/2 + n + 1}{n}.
		\]
		If $X$ is of type $\text{Kum}_n$, then the Huybrechts--Riemann--Roch polynomial takes the form
		\[
		\RR_X(t) = (n+1)\binom{t/2 + n}{n}.
		\]
		We will say that the Huybrechts--Riemann--Roch polynomial is of $\operatorname{K3}^{[n]}$-type or $\text{Kum}_n$-type if it corresponds to one of the two examples above. In \cite{Rios2020} it is proven that the Huybrechts--Riemann--Roch polynomials for the deformation class of OG6 and OG10 are of $\text{Kum}_3$-type and $K3^{[5]}$-type respectively.
	\end{exmp}
	
	%%%%%%%%%%%%%%%%%%%%%%%%%%%%%%
	%%%%%%%%%%%%%%%%%%%%%%%%%%%%%%
	\section{Lagrangian fibrations}\label{section:LagrangianFibrations}
	Throughout this section $X$ is a primitive symplectic variety of dimension $2n$.
	
	\begin{defn}
		Let $X$ be a primitive symplectic variety of dimension $2n$.
		\begin{enumerate}
			\item A subvariety $Z\subset X$ of dimension $n$ is called \emph{lagrangian} if $Z\cap X_{\operatorname{reg}}\neq\emptyset$ and $\sigma_{\operatorname{reg}}|_{Z_{\operatorname{reg}}\cap X_{\operatorname{reg}}}=0$.
			
			\item A surjective morphism $f\colon X\to B$ with connected fibres onto a normal K\"ahler space of dimension $n$ is a \emph{lagrangian fibration} if the general fibre of $f$ is a lagrangian subvariety.
		\end{enumerate}
	\end{defn}
	
	The following result is originally due to Schwald (see \cite{Schwald2020}); we state here the slightly more general version in \cite{KamenovaLehn2024}.
	
	\begin{thm}[\protect{\cite[Theorem 2.8]{KamenovaLehn2024}}]\label{thm:KL instead of Sch}
		Let $X$ be a primitive symplectic variety of dimension $2n$ and $f\colon X\to B$ a surjective morphism with connected fibres onto a normal K\"ahler space.
		
		Then $f\colon X\to B$ is a lagrangian fibration and
		\begin{enumerate}
			\item $B$ is a $\Q$-factorial projective klt variety of Picard rank $1$;
			\item the general fibre of $f$ is an abelian variety of dimension $n$ completely contained in the smooth locus of $X$;
			\item $f$ is equidimensional and all irreducible components of each fibre are lagrangian subvarieties.
		\end{enumerate}
		If moreover $X$ is irreducible symplectic, then $B$ is Fano. In this case, if $B$ is smooth, then $B\cong\P^n$.
	\end{thm}
	
	Notice that the claim of the theorem above is that the general fibre of a lagrangian fibration is projective even if $X$ is not. The following lemma is essentially \cite[Lemma~2.2]{Matsushita:Deformations2016} (see also \cite[Lemma~1.5]{Voisin1992}). In the singular setting it follows from \cite[Theorem~1.1]{FSY:IntersectionCohmologyLagrFibr}, but we provide the details of the proof for completeness.
	
	\begin{lem}\label{lem:restrictionVoisinsingular}
		In the setting of Theorem~\ref{thm:KL instead of Sch}, let $X_b$ be a smooth fibre of $f$ and let $L:=f^{*}\sO_B(1)\in \oH^2(X,\Z)$. If $r_b\colon\oH^2(X,\Z)\to\oH^2(X_b,\Z)$
		is the restriction map, then 
		\[ \ker(r_b)=L^\perp, \]
		where the perpendicular is taken with respect to the BBF form on $X$. In particular, the image of the restriction map is of rank $1$ and is generated by an ample class on $X_b$.
	\end{lem}
	\begin{proof}
		The restriction $r_b\colon\oH^2(X,\C)\to \oH^2(X_b,\C)$ is a morphism of pure weight--two Hodge structures and, if $\sigma_X$ is a symplectic form on $X$, we have that $r_b(\sigma_X) = 0$. Therefore, $\mathrm{Im}(r)\subseteq \oH^{1,1}(X,\C)$. Now, if $\omega\in \oH^{2}(X,\C)$ is a K\"ahler class on $X$, then since $X_b\subset X_{\mathrm{reg}}$ we have that $r_b(\omega)$ is a K\"ahler class in $\oH^{2}(X_b,\C)$. By the Hodge--Riemann bilinear relations and the Hard Lefschetz Theorem, if $\alpha\in \oH^2(X,\C)$ satisfies 
		\[
		\int_{X_b}r_b(\alpha)\smile r_b(\omega)^{n-1} = \int_{X_b}r_b(\alpha)^2\smile r_b(\omega)^{n-2} = 0,
		\]
		then $r_b(\alpha) = 0$. Let $s$ and $t$ be formal variables. By the Fujiki relations (cf.\ \cref{prop:Fujiki}) we get the following
		\[
		c_Xq_X(\alpha + s\omega + tL)^n = \int_X (\alpha + s\omega + tL)^{2n}.  
		\]
		By comparing the $s^{n-1}t^n$ and $s^{n-2}t^n$ terms on both sides we get $r(\alpha) = 0$ if and only if $q(\alpha,L) = 0$.  
	\end{proof}

	%%%%%%%%%%%%%%%%%%%%%%%%%%%%%
	\subsection{Deformations of lagrangian fibrations}
	
	Let us start with the main definition.
	\begin{defn}\label{defn:def of lagr fibr}
		Let $p\colon\mathcal{X}\to T$ be a locally trivial family of primitive symplectic varieties. Then we say that it is a \emph{locally trivial family of lagrangian fibrations} if there exists a commutative diagram 
		\[ 
		\xymatrix{
			\mathcal{X}\ar[rr]^{f}\ar[dr]_{p} &   & \mathcal{B} \ar[dl]^{s} \\
			& T &
		}
		\]
		such that
		\begin{itemize}
			\item $f$ is a $T$-morphism;
			\item $s$ is projective;
			\item for every $t\in T$, the restriction $f_t\colon\mathcal{X}_t\to\mathcal{B}_t$ is a lagrangian fibration.
		\end{itemize}
		
		By abuse of notation, we denote by $p\colon\mathcal{X}/\mathcal{B}\to T$ a locally trivial family of lagrangian fibrations.
	\end{defn}
	
	We will say that two lagrangian fibrations $f_i\colon X_i\to B_i$ are \emph{locally trivial deformation equivalent} if there exists a locally trivial family of lagrangian fibrations $p\colon\mathcal{X}/\mathcal{B}\to T$ where $T$ is connected and $f_i$ belong to the family.
	\medskip
	
	We now recall a result about infinitesimal deformations of a lagrangian fibration $f\colon X\to B$. Let $A$ be an ample line bundle on $B$ and put $L=f^*A$, which is a semiample line bundle on $X$ such that $q_X(L)=0$ (cf.\ Proposition~\ref{prop:Fujiki}).
	
	With notation as in Section~\ref{section:inf Torelli}, we denote by $(\mathcal{X}_L,\mathcal{L})$ the universal Kuranishi family of the pair $(X,L)$. Moreover, we denote by $\pi_L\colon\mathcal{X}_{L}\to\Deflt(X,L)$ the projection. 
	
	\begin{prop}[\protect{\cite[Theorem~A.1, Theorem~A.2]{EFGMS:Boundedness}}]\label{prop:higher direct images are free}
		Let $f\colon X\to B$, $L$, $\pi_L$ and $\mathcal{L}$ as above. Then:
		\begin{enumerate}
			\item up to shrinking $\Deflt(X,L)$, the higher direct sheaves $R^i\pi_{L*}\mathcal{L}$ are locally free for every $i\geq0$;
			\item up to shrinking $\Deflt(X,L)$, there is a locally trivial deformation of lagrangian fibrations 
			\[ 
			\xymatrix{
				\mathcal{X}_L\ar[rr]^{\tilde{f}}\ar[dr]_{\pi_L} &   & \mathbb{P}(\pi_{L*}\mathcal{L}) \ar[dl]^{s_L} \\
				& \Deflt(X,L) &
			}
			\]
			such that the fibre over the reference point of $\tilde{f}$ is the lagrangian fibration $f$.
		\end{enumerate}
	\end{prop}
	
	\begin{rem}
		The same result, for polarised families, has been proved by Matsushita in \cite{Matsusecret}.
	\end{rem}
	
	One of the main ingredients in the proof of Proposition~\ref{prop:higher direct images are free} is the following result, which we will need later. It is a consequence of \cite[Theorem~0.9]{DP:CharofKahler} after having taken a simultaneous resolution of singularities.
	
	\begin{lem}[\protect{\cite[Lemma~A.12 and proof of Proposition~A.13]{EFGMS:Boundedness}}]\label{lemma:molto general is semiample}
		Let $p\colon \cX\to S$ be a locally trivial family of primitive symplectic varieties and $\mathcal{L}$ a line bundle on $\cX$, flat over $S$. Assume that $S$ is irreducible and that there exists $t_0\in S$ such that $\mathcal{L}_{t_0}$ is nef. Then $\mathcal{L}_t$ is nef for every $t$ outside a countable union of irreducible subvarieties.
	\end{lem}
	
	In particular, $\mathcal{L}_t$ is nef for every $t$ such that $\Pic(\mathcal{X}_t)$ is cyclic, generated by $\mathcal{L}_t$.

	%%%%%%%%%%%%%%%%%%%%%%%%%%%%%%%%%%
	\subsection{The SYZ conjecture}\label{section:SYZ}
	
	Recall that a line bundle $L$ on a compact K\"ahler space $X$ is \emph{nef} if it belongs to the closure of the K\"ahler cone.
	
	\begin{rem}\label{rmk:quasi-nef iff nef}
		Let $X$ be a normal compact K\"ahler space and $p\colon\widetilde X \to X$ a resolution of singularities.
		As a consequence of \cite[Lemma~2.38]{DHP:MMP4Complex}, a line bundle $L$ on $X$ is nef if and only if $p^*L$ is nef on $\widetilde X$ (see also \cite[Proposition~2.7]{NAKAYAMA:LowerSemiContinuity} for projective varieties).
	\end{rem}
	
	Let $X$ be a primitive symplectic variety. If $f\colon X\to B$ is a lagrangian fibration and $A$ is an ample line bundle on $B$, then $f^*A$ is semiample, hence nef, and $q_X(f^*A)=0$. The SYZ conjecture predicts that the converse holds.
	
	\begin{conjecture}[SYZ conjecture for primitive symplectic varieties]
		Let $X$ be a primitive symplectic variety and $L$ a line bundle on it. If $L$ is nef and $q_X(L)=0$, then there exists a lagrangian fibration $f\colon X\to B$ and a very ample line bundle $A$ on $B$ such that $L^{\otimes k}=f^*A$.
	\end{conjecture}
	
	If $X$ is smooth and belongs to one of the known deformation types, then the conjecture holds true, see: \cite{BM:MMP,Markman:LagrangianFibrations,Matsushita:Isotropic2017,Wieneck2016} for the $\operatorname{K3}^{[n]}$ case; \cite{Yoshioka:Bridgeland,Wieneck2018} for the $\operatorname{Kum}_n$ case; \cite{MR} for the OG6 case; and \cite{MO} for the OG10 case. Moreover, it has been proved for fourfolds satisfying some topological conditions in \cite{DHMV2024}.
	
	%%%%%%%%%%%%%%%%%%%%%%%%%%%%%%
	%%%%%%%%%%%%%%%%%%%%%%%%%%%%%%
	\section{Moduli spaces of lagrangian fibrations}\label{section: moduli spaces of lagrangian fibrations}
	
	In this section we study lagrangian fibrations over a fixed base $B$. By Theorem~\ref{thm:KL instead of Sch}, we know that $B$ is $\Q$-factorial projective of Picard rank $1$, so that $\Pic(B)_\Q$ is generated by an ample invertible sheaf $\sO_B(1)$. If we denote by $m_B\in \mathbb{N}$ the smallest positive number such that $\sO_B(m_B)\in\Pic(B)$ is a very ample line bundle, then we put 
	\[ h_B:=\sO_B(m_B). \]
	For example (cf.\ \cite[Theorem~1.22]{MP25}): if $B=\P^n$, then $h_B=\sO_{\P^n}(1)$; if $B=\P(1^{15},2^6,3)$ (cf.\ \cite{LLX:CubicFivefold}), then $h_B=\sO_{\P(1^{15},2^6,3)}(6)$; if $B=\P(1,1,2)$ (cf.\ \cite{Nanni:SYZNikulin}), then $h_B=\sO_{\P(1,1,2)}(2)$.
	
	\begin{defn}\label{defn:L induces a lagr fibr}
		Let $X$ be a primitive symplectic variety and $L$ a line bundle on $X$ such that $q_X(L)=0$. Then we say that $L$ \emph{defines a lagrangian fibration over} $B$ if there exists a lagrangian fibration $f\colon X\to B$ such that $f^*h_B=L$.
	\end{defn}
	
	The purpose of this section is to prove the following theorem. 
	
	\begin{thm}\label{thm:main result on lagr}
		For $i=1,2$, let $X_i$ be a primitive symplectic variety and let $L_i\in\Pic(X_i)$ be a nef divisor with $q_{X_i}(L_i)=0$.
		Suppose that $L_1$ defines a lagrangian fibration on $X_1$ over $B$. 
		
		If there exists a locally trivial parallel transport operator 
		\[ \mathsf{P}\colon\oH^2(X_1,\Z)\to\oH^2(X_2,\Z) \]
		such that $\mathsf{P}(L_1)=L_2$, then $L_2$ defines a lagrangian fibration on $X_2$ over $B$. 
	\end{thm}
	
	Notice that in Theorem~\ref{thm:main result on lagr} we do not know a priori whether the corresponding lagrangian fibrations are locally trivial deformation equivalent according to Definition~\ref{defn:def of lagr fibr}. The next result proves this deformation invariance, provided the varieties have milder singularities.
	
	\begin{thm}\label{thm:def as lagr fibr iff existence of suitable pto}
		Let $X_1$ and $X_2$ be two locally trivial deformation equivalent primitive symplectic varieties with terminal and $\Q$-factorial singularities. For $i=1,2$, let $f_i\colon X_i\to B$ be two lagrangian fibrations over $B$. 
		
		Then $f_i\colon X_i\to B$ are locally trivial deformations as lagrangian fibrations over $B$ if and only if there exists a locally trivial parallel transport operator
		\[ \mathsf{P}\colon \oH^2(X_1,\Z)\longrightarrow\oH^2(X_2,\Z) \]
		such that $\mathsf{P}(f_1^*h_B)=f_2^*h_B$.
	\end{thm}
	
	The hypotheses of terminality and $\Q$-factoriality are needed to settle a technical intermediate step in our proof, but alternative proofs could give an unconditional statement (see also Remark~\ref{rmk:term+Q-fact}).
	
	%%%%%%%%%%%%%%%%%%%%%%%%%%%%%%%%%%%%%%%%
	\subsection{Some preliminary results}
	
	The following is well-known.
	
	\begin{lem}\label{lemma:canonical bundle trivial}
		Let $p\colon\mathcal{X}\to\Delta$ be a locally trivial family of primitive symplectic varieties over the unit disc. Then the canonical bundle of $\mathcal{X}$ is trivial.
	\end{lem}
	\begin{proof}
		First of all, let us remark that $\mathcal{X}$ is normal and Gorenstein. In fact, being $p\colon\mathcal{X}\to\Delta$ locally trivial, both properties follow from the fact that the fibres of $p$ are normal and Gorenstein, and $\Delta$ is smooth.
		Therefore it is enough to exhibit a dense open subset of $\mathcal{X}$ whose boundary has codimension at least $2$ and whose canonical bundle is trivial.
		
		Since $p\colon\mathcal{X}\to\Delta$ is locally trivial, there is a smooth fibration $p_0\colon\mathcal{X}_{\operatorname{sm}}\to\Delta$ whose fibres are the smooth loci of the fibres of $p$. Clearly $\mathcal{X}_{\operatorname{sm}}$ is a dense open subset of $\mathcal{X}$ with boundary of codimension at least $2$. Moreover, being $p_0$ smooth, the triviality of the canonical bundle follows from the relative tangent short exact sequence of $p_0$.
	\end{proof}
	
	The following is the singular analogue of \cite[Lemma~3.2]{Matsushita:Isotropic2017}, which is a special case of \cite[Corollary~3.14]{NAKAYAMA:LowerSemiContinuity}. Notice that the same result has been used in \cite[Appendix~A]{EFGMS:Boundedness} for the proof of Proposition~\ref{prop:higher direct images are free}.
	
	\begin{lem}\label{lemma:Nakayama}
		Let $p\colon\mathcal{X}\to\Delta$ be a locally trivial family of primitive symplectic varieties over the unit disc. Suppose that there exists a line bundle $\mathcal{L}$ on $\mathcal{X}$, flat over $\Delta$. If
		\begin{enumerate}
			\item $\mathcal{L}_t$ is semiample for every $t\neq0$; and
			\item $\mathcal{L}_0$ is quasi-nef,
		\end{enumerate}
		then, up to shrinking $\Delta$, the sheaf $R^ip_*\mathcal{L}^k$ is locally free for every $i\geq0$ and every $k\geq1$.
		
		Moreover, the natural morphism
		\[ R^ip_*\mathcal{L}^k\otimes k(0)\to \oH^i(\mathcal{X}_0,\mathcal{L}_0^k) \]
		is an isomorphism for every $i\geq0$ and every $k\geq1$.
	\end{lem}
	Recall that a line bundle on $X$ is \emph{quasi-nef} if there exists a resolution of singularities $f\colon\widetilde X\to X$ such that $f^*L$ is nef. It is nowadays known that $L$ is quasi-nef if and only if it is nef (see \cite[Lemma~2.38]{DHP:MMP4Complex}), but we keep the same terminology as in \cite{NAKAYAMA:LowerSemiContinuity} for consistency.
	\begin{proof}    
		By \cite[Proposition~2.17]{NAKAYAMA:LowerSemiContinuity}, up to shrinking $\Delta$, there exists a commutative diagram
		\[ 
		\xymatrix{
			\mathcal{X}\ar[d]^{p} & Y\ar[l]^{f}\ar[d]_{h} \\
			\Delta & Z\ar[l]_{g}
		}
		\]
		where 
		\begin{itemize}
			\item $Y$ and $Z$ are complex manifolds,
			\item $f$ is proper and birational, $h$ is proper with connected fibres and $g$ is projective,
			\item there exists a divisor $H$ on $Z$ such that
			\begin{itemize}
				\item $H_t$ is nef for all $t\in\Delta$ and big and nef for all $t\neq0$;
				\item $h^*H=f^*\mathcal{L}$.
			\end{itemize}
		\end{itemize} 
		Clearly, for every $k\geq1$ we have $h^*H^{\otimes k}=f^*\mathcal{L}^{\otimes k}$. Since $H_t$ (and hence $H_t^{\otimes k}$) is big and nef for general $t$, there exists an effective divisor $E$ on $Z$ such that $H^{\otimes k}\otimes\sO_Z(-\delta E)$ is relatively ample with respect to $g$ for all $k\geq1$. Here $\delta$ is a rational number such that $0<\delta\ll 1$.
		Arguing now as in the second part of the proof of \cite[Corollary~3.14]{NAKAYAMA:LowerSemiContinuity}, we can replace $\cX$ by $Y$ and $\mathcal{L}$ (resp.\ $\mathcal{L}^{\otimes k}$) with $H\otimes\sO_Z(-\delta E)$ (resp.\ $H^{\otimes k}\otimes\sO_Z(-\delta E)$) and the proof follows from \cite[Lemma~3.9]{NAKAYAMA:LowerSemiContinuity} and Lemma~\ref{lemma:canonical bundle trivial}.
		
		Finally, the last part of the statement follows from the Andreotti--Grauert Theorem.
	\end{proof}
	
	Next, let us recall the following result of Matsushita, see \cite[Lemma 3.1]{Matsushita:Isotropic2017}. In loc.\ cit.\ Matsushita states the result only for lagrangian fibrations over $\P^n$, but it holds more generally. We use the following notation: for a line bundle $M$ on a variety $Y$,
	\[ R(Y,M)=\bigoplus_{k\geq0}\oH^0(Y,M^{\otimes k}) \]
	stands for the associated graded algebra of global sections.
	
	\begin{lem}[\protect{\cite[Lemma 3.1]{Matsushita:Isotropic2017}}]\label{lemma:semiample if and only if nef +}
		Let $L$ be a line bundle on a primitive symplectic variety $X$. Then $L$ defines a lagrangian fibration over $B$ if and only if $L$ is nef and 
		\begin{equation}\label{eqn:char of lagr over B} \
			R(X,L)=R(B,h_B). \end{equation}
	\end{lem}
	\begin{proof}
		First of all, by (\ref{eqn:char of lagr over B}) we have that $R(X,L)$ is finitely generated in degree $1$. We claim that $L$ is semiample. In fact, let us put $\dim X=2n$. Then equality
		(\ref{eqn:char of lagr over B}) implies that the Iitaka dimension of $L$ is $n$; moreover, by the Fujiki relation (see Proposition~\ref{prop:Fujiki}), the numerical Kodaira dimension of $L$ is also $n$, so that it is abundant. By Kawamata's Theorem (see \cite[Theorem~1.1]{Fujino:Kawamata}), it follows that $L$ is semiample. Now, the morphism $\varphi_L\colon X\to|L|$ is regular and its image is $\operatorname{Proj}R(X,L)\cong\operatorname{Proj}R(B,h_B)=B$. Finally, we claim that $\varphi_L$ has connected fibres, thus concluding the proof.
		The Stein factorisation of $\varphi_L$ is
		\[ \xymatrix{
			X \ar@{->}[r]^-{\varphi_L} \ar@{->}[dr]_-{\varphi_{mL}} & \operatorname{Im}(\varphi_L) \ar@{^{(}->}[r] &  |L| \\
			& \operatorname{Im}(\varphi_{mL}) \ar@{->}[u]_{g} \ar@{^{(}->}[r] &  |mL| 
		} \]
		where $\varphi_{mL}$ has connected fibres for $m\geq0$ and $g$ is the restriction of the rational projection from $|mL|$ to $|L|$. Notice that $|kL|=|kh_B|$ for every $k>0$ by (\ref{eqn:char of lagr over B}). Since $h_B$ is very ample, we have that  $\operatorname{Im}(\varphi_{mL})=\operatorname{Im}(\varphi_{L})=B$ and $g$ is the inverse of the Veronese embedding. 
		Therefore $g$ is an isomorphism and $\varphi_L$ has connected fibres as claimed.
	\end{proof}
	
	We apply Lemma~\ref{lemma:semiample if and only if nef +} in the following way.
	
	\begin{lem}\label{lemma:base B is open}
		Let $X$ be a primitive symplectic variety and $f\colon X\to B$ a lagrangian fibration. Put $L=f^*h_B$.
		Then there exists an open subset $U\subset\Deflt(X,L)$ containing the reference point and such that for every $t\in U$ the line bundle $\mathcal{L}_t$ on $\mathcal{X}_t$ induces a lagrangian fibration over $B$.
	\end{lem}
	\begin{proof}
		Define the sheaf of graded $\sO_{\Deflt(X,L)}$-algebras 
		\[ \mathcal{R}(\mathcal{X},\mathcal{L})=\bigoplus_{k\geq0}p_*\mathcal{L}^{\otimes k}. \] 
		By Lemma~\ref{lemma:Nakayama}, there exists an open subset $U\subset\Deflt(X,L)$ containing the reference point such that the restriction $\mathcal{R}(\mathcal{X},\mathcal{L})|_U$ is a locally free sheaf of graded $\sO_U$-algebras. Moreover $\mathcal{R}(\mathcal{X},\mathcal{L})_t=R(\cX_t,\mathcal{L}_t)$ for every $t\in U$. Then the claim follows from Lemma~\ref{lemma:semiample if and only if nef +}.
	\end{proof}
	
	Now let $X$ be a primitive symplectic variety and $L$ an isotropic and nef line bundle on $X$ (notice that we do not suppose that $L$ is induced by a lagrangian fibration here). As in Section~\ref{section:inf Torelli}, let us denote by $(\mathcal{X}_L,\mathcal{L})$ the universal family over the Kuranishi space $\Deflt(X,L)$ of infinitesimal locally trivial deformations of the pair $(X,L)$. Inside $\Deflt(X,L)$ we can define the following subspace (see Definition~\ref{defn:L induces a lagr fibr}),
	\[ \Deflt(X,L)_{\operatorname{lagr},B}=\{t\in\Deflt(X,L)\mid \mathcal{L}_t \mbox{ defines a lagrangian fibration over }B \}\,. \]
	
	\begin{prop}\label{prop:Def lagr B = Def(X,L)}
		If $\Deflt(X,L)_{\operatorname{lagr},B}\neq\emptyset$, then $\Deflt(X,L)_{\operatorname{lagr},B}\subset\Deflt(X,L)$ is open and dense.
	\end{prop}
	\begin{proof}
		We closely follow \cite[Lemma~3.4]{Matsushita:Isotropic2017}.
		First of all, $\Deflt(X,L)_{\operatorname{lagr},B}$ is open in $\Deflt(X,L)$ by Lemma~\ref{lemma:base B is open}.
		
		Next, let us prove that $\Deflt(X,L)_{\operatorname{lagr},B}$ is dense in $\Deflt(X,L)$. Put $\overline{\Deflt}(X,L)_{\operatorname{lagr},B}$ for the closure of $ \Deflt(X,L)_{\operatorname{lagr},B}$ in $\Deflt(X,L)$.
		
		Let $t\in\overline{\Deflt}(X,L)_{\operatorname{lagr},B}$ be a point such that $\oH^{1,1}(\mathcal{X}_{t},\Q)=\Q\mathcal{L}_t$. We call such a point very general and we notice that it exists, since the set of points corresponding to varieties with Picard rank $1$ is dense. We claim that $t\in\Deflt(X,L)_{\operatorname{lagr},B}$, thus concluding the proof.
		
		First of all, since $t\in\overline{\Deflt}(X,L)_{\operatorname{lagr},B}$ is such that $\oH^{1,1}(\mathcal{X}_{t},\Q)=\Q\mathcal{L}_t$, then $\mathcal{L}_t$ is nef by Lemma~\ref{lemma:molto general is semiample}. 
		Let us now take a small disc $\Delta\subset \Deflt(X,L)$ such that $t\in\Delta$ and $\Delta\setminus\{t\}\subset\Deflt(X,L)_{\operatorname{lagr},B}$. If we denote by $\pi_\Delta\colon\mathcal{X}\to\Delta$ the restriction of the Kuranishi family, then by Lemma~\ref{lemma:Nakayama}, up to shrinking $\Delta$, we have that $\pi_{\Delta*}\mathcal{L}^{\otimes k}$ is locally free for every $k\geq1$. Therefore the sheaf of graded $\sO_U$-algebras $\mathcal{R}(\cX,\mathcal{L})=\bigoplus_{k\geq0}\pi_{\Delta,*}\mathcal{L}^{\otimes k}$ is locally free and for every $s\in\Delta$ there is an equality
		\[ \mathcal{R}(\cX,\mathcal{L})_s=R(\cX_s,\mathcal{L}_s). \]
		Combining with Lemma~\ref{lemma:semiample if and only if nef +} and taking $s=t$, we get that $\mathcal{L}_t$ induces a lagrangian fibration over $B$, i.e.\ $t\in \Deflt(X,L)_{\operatorname{lagr},B}$.
	\end{proof}
	
	\begin{rem}\label{remark:complement of hypersurfaces}
		Arguing as in the last part of \cite[Lemma~3.4]{Matsushita:Isotropic2017}, we have that 
		\[ \Deflt(X,L)\setminus\Deflt(X,L)_{\operatorname{lagr},B}\subset\left\{ t\in\Deflt(X,L)\mid \dim \oH^{1,1}(\mathcal{X}_t,\Z)\geq 2 \right\}, \] 
		where the latter is a countable union of hypersurfaces.
	\end{rem}
	
	%%%%%%%%%%%%%%%%%%%%%%%%%%%%%%%%%%%%%%%%%%%
	\subsection{Proof of Theorem~\ref{thm:main result on lagr}}
	
	Following the notation introduced in Section \ref{section:marked moduli space}, let us consider the moduli space $\fM_\ell$, where $\ell\in\Lambda$ is an isotropic class. Recall that $\fM_\ell$ parametrises marked pairs $(X,\eta)$ where $L=\eta^{-1}(\ell)$ is of type $(1,1)$.  
	
	From now on we work with a connected component $\mathfrak{M}_\Lambda^0$ of $\fM_\Lambda$. Recall from Section~\ref{section:orientation} that the choice of $\mathfrak{M}_\Lambda^0$ determines a connected component $\Omega_\ell^+$ of $\Omega_\ell$ and we put $\mathfrak{M}_\ell^0=\mathcal{P}_0^{-1}(\Omega_\ell^+)$.
	
	Define the subsets
	\[ \fM_\ell^{\operatorname{nef}}:=\left\{ (X,\eta)\in\fM_\ell^0\mid \eta^{-1}(\ell) \mbox{ is nef} \right\} \]
	and 
	\[ \fM_\ell^{\operatorname{lagr},B}:=\left\{ (X,\eta)\in\fM_\ell^0\mid \eta^{-1}(\ell)  \mbox{ defines a lagrangian fibration over } B \right\}. \]
	
	\begin{rem}\label{rmk:lagr is open in L}
		Let $f\colon X\to B$ be a lagrangian fibration and $L=f^*h_B$. As already remarked, $L$ is semiample, hence nef. In particular $\fM_\ell^{\operatorname{lagr},B}\subset \fM_\ell^{\operatorname{nef}}$.
	\end{rem}
	
	Theorem~\ref{thm:main result on lagr} is a consequence of the following result.
	
	\begin{thm}\label{thm:nef=lagr}
		Let $\fM_\ell^{\operatorname{nef}}$ and $\fM_\ell^{\operatorname{lagr},B}$ be as above.
		If $\fM_\ell^{\operatorname{lagr},B}\neq\emptyset$, then $\fM_\ell^{\operatorname{lagr},B}\subset\fM^0_\ell$ is open and dense. Moreover, in this case, we have an equality 
		\[ \fM_\ell^{\operatorname{lagr},B}=\fM_\ell^{\operatorname{nef}}. \]
	\end{thm}
	\begin{proof}
		The fact that $\fM_\ell^{\operatorname{lagr},B}\subset\fM^0_\ell$ is open and dense follows from Proposition~\ref{prop:Def lagr B = Def(X,L)}.
		Moreover, we have already remarked that $\fM_\ell^{\operatorname{lagr},B}\subset \fM_\ell^{\operatorname{nef}}$, so it remains to prove the other inclusion.
		We closely follow the proof of \cite[Claim~3.2]{Matsushita:Isotropic2017}.
		
		Let $(X,\eta)\in\fM_\ell^{\operatorname{nef}}$ and put $L=\eta^{-1}(\ell)$. Working locally around $(X,\eta)$, let us consider the Kuranishi space $\Deflt(X,L)$ and denote by $0$ the reference point. We need to prove that $0\in\Deflt(X,L)^{\operatorname{lagr},B}=\Deflt(X,L)\cap\fM_\ell^{\operatorname{lagr},B}$.
		
		Since $\fM_\ell^{\operatorname{lagr},B}$ is dense in $\fM^0_\ell$, we have that $0\in\Deflt(X,L)$ belongs to the closure of $\Deflt(X,L)^{\operatorname{lagr},B}$. Let us then choose a small disc $\Delta\subset\Deflt(X,L)$ such that $0\in\Delta$ and $\Delta\setminus\{0\}\subset\Delta\cap\Deflt(X,L)^{\operatorname{lagr},B}$.
		
		As in Section~\ref{section:inf Torelli}, let us denote by $(\mathcal{X}_L,\mathcal{L})$ the universal family over $\Deflt(X,L)$. By abuse of notation, we keep the same notation for their restrictions to the disc $\Delta$.
		
		Let us then consider the projection $\pi\colon\mathcal{X}_L\to\Delta$. By Lemma~\ref{lemma:Nakayama} we have that $\pi_*\mathcal{L}^{\otimes k}$ is locally free with fibre over $t\in\Delta$ isomorphic to $\oH^0(\mathcal{X}_t,\mathcal{L}_t^{\otimes k})$.
		
		When $t\neq0$, since $\mathcal{L}_t$ induces a lagrangian fibration over $B$ by assumption, we have that $R(\mathcal{X}_t,\mathcal{L}_t)=R(B,h_B)$. Therefore the same must be true for $t=0$ and by Lemma \ref{lemma:semiample if and only if nef +} we conclude that $0\in\Deflt(X,L)^{\operatorname{lagr},B}$, that is $(X,\eta)\in\fM_\ell^{\operatorname{lagr},B}$.
	\end{proof}
	
	\proof[Proof of Theorem~\ref{thm:main result on lagr}]
	First of all, let us choose a marking $\eta_2$ of $X_2$ and let us put $\eta_1=\eta_2\circ\mathsf{P}$. Since $\mathsf{P}(L_1)=L_2$, the pairs $(X_1,\eta_1)$ and $(X_2,\eta_2)$ belong to the same connected component $\fM^0_\ell$ of $\fM_\ell$, and since $L_1$ and $L_2$ are nef we further have that $(X_1,\eta_1),(X_2,\eta_2)\in\fM_\ell^{\operatorname{nef}}$. Here $\ell\in\Lambda$ is the element such that $\eta_1(L_1)=\ell=\eta_2(L_2)$. 
	
	Now, since by assumption $(X_1,L_1)\in\fM_\ell^{\operatorname{lagr},B}$, the latter is non-empty and by Theorem~\ref{thm:nef=lagr} we have $\fM_\ell^{\operatorname{nef}}=\fM_\ell^{\operatorname{lagr},B}$. Therefore also $(X_2,\eta_2)\in\fM_\ell^{\operatorname{lagr},B}$, i.e.\ $L_2$ induces a lagrangian fibration with base $B$.
	\endproof
	
	%%%%%%%%%%%%%%%%%%%%%%%%%%%%%%%%
	\subsection{Proof of Theorem~\ref{thm:def as lagr fibr iff existence of suitable pto}}
	
	We start with the following remark, which will be useful for the proof.
	
	\begin{prop}\label{lemma: M ell nef è connesso}
		Assume that all the varieties parametrised by $\fM_\Lambda$ have terminal and $\Q$-factorial singularities. Let $\ell\in\Lambda$ be an isotropic class. If $\fM_\ell^{\operatorname{lagr},B}\neq\emptyset$, then the locus
		\[ \mathcal{W}=\left\{ (X,\eta)\in\fM_\ell^{\operatorname{lagr},B}\mid \dim\, \oH^{1,1}(X,\Z)=1 \right\} \]
		is path-connected.
	\end{prop}
	
	In order to prove the proposition, we need to settle two intermediate steps. The first one is essentially \cite[Theorem~2.4.(5)]{OguisoBimeromorphisAutomorphismGroups}. We reproduce the proof for the reader's convenience. 
	
	\begin{lem}\label{lem:automorphismsParabolicHodgeStructure}
		Let $(X,\eta)$ be a marked pair such that $\Pic(X)=\Z L$ with $q_X(L)=0$. Then $\operatorname{Aut}_{\mathrm{Hdg}}(\oH^2(X,\Z))=\pm\operatorname{id}$
	\end{lem}
	\begin{proof}
		Let $\varphi\in\Aut_{\Hdg}(\oH^2(X,\Z))$ be an automorphism, and let $\sigma_X\in\oH^{2,0}(X)$ be the symplectic form. Up to composing $\varphi$ with $-\operatorname{id}$, we can assume that $\varphi$ preserves the orientation (see \cref{section:orientation} for the notion of orientation). 
		Moreover, since $q_X(L)=0$ by assumption, we can also assume that $L$ belongs to the boundary of the \emph{positive cone} $\mathcal{C}$, i.e.\ the cone of positive classes in $\oH^{1,1}(X,\R)$ containing the K\"ahler cone. 
		
		Now, since $\varphi$ is orientation--preserving, we have that $\varphi(\mathcal{C})=\mathcal{C}$ and therefore $\varphi(L)=L$. 
		Moreover, if $T(X)$ denotes the transcendental lattice of $X$, i.e.\ the smallest sub-Hodge structure containing the symplectic form $\sigma_X$, by hypothesis we must have $\Pic(X)\cap T(X)=L$ (recall that the transcendental lattice is orthogonal to the Picard lattice). By minimality of $T(X)$, we further have that the characteristic polynomial of $\varphi|_{T(X)}$ is a power of an irreducible polynomial (cf.\ \cite[Theorem~2.4.(1)]{OguisoBimeromorphisAutomorphismGroups}). This implies that $\varphi(\sigma_X)=\sigma_X$, therefore finishing the proof. 
	\end{proof}
	
	The second one can be seen as a slight strengthening of \cite[Corollary~6.15]{BakkerLehn2022}.
	
	\begin{lem}\label{lemma:cyclic pic is iso}
		Let $(X,\eta)$ and $(X',\eta')$ be two marked primitive symplectic varieties in the same connected component of $\fM_{\Lambda}$. Let us assume that $X$ and $X'$ are terminal and $\Q$-factorial. Moreover assume $\Pic(X)=\Z L$ and $\Pic(X')=\Z L'$ with $q_X(L)=0$ and $q_{X'}(L')=0$. If $\cP(X,\eta)=\cP(X',\eta')$, then $(X,\eta)=(X',\eta')$.
	\end{lem}
	\begin{proof}
		By \cite[Theorem~6.14]{BakkerLehn2022} there exists a bimeromorphic map $f\colon X \dashrightarrow X'$. Using  \cite[Lemma~4.2]{GolotaRigidityLemma} we get that $f$ is an isomorphism in codimension one. By Chow's Lemma there exists a resolution of indeterminacies
		\[
		\begin{tikzcd}
			&  Y\ar[dl,swap,"p"]\ar[dr,"q"]\\
			X\ar[rr,dashed,"f"] & & X'
		\end{tikzcd}
		\]
		where $q$ is projective and a sequence of blowups with smooth centres. Let $E=\sum_i E_i$ be the exceptional divisor. Since $f$ does not contract divisors, then  $E$ is exceptional both for $p$ and $q$. If for every $i$ and every curve $C\subset E_i$ we have that $C$ gets contracted by $q$ to a point, then by the rigidity Lemma \cite[Lemma~4.1]{GolotaRigidityLemma} we will have that $f^{-1}$ is a morphism; exchanging the roles of $X$ and $X'$, we get that $f$ is an isomorphism. 
		
		We can therefore assume that there exists some $C\subseteq E_i$ that is contracted to a point by $p$ but not by $q$.
		Let $\alpha\in\oH^{1,1}(X)$ be a Kähler class. Then by \cite[Lemma~4.4]{GolotaRigidityLemma} we have
		\[
		q^*q_*p^*\alpha - p^*\alpha = \sum_i a_i E_i
		\]
		with $a_i\geq 0$. We compute 
		\begin{align*}
			f_*\alpha. q(C)=q_*p^*\alpha.q(C) = &  \,\deg(q\vert_C)( q^*q_*p^*\alpha).q(C) \\
			= & \, \deg(q\vert_C)(p^*\alpha + E).C \\ 
			= & \, 0 - \deg(q\vert_C)a_i\leq 0\, .
		\end{align*}
		
		Now, let $M$ be the BBF-dual of $q(C)$. Recall that $M$ is uniquely determined by the property that $q_{X'}(\beta,M)=\int_{q(C)}\beta$, for every $\beta\in\oH^2(X',\Q)$. Notice also that we have that $M$ is a $(1,1)$-class (cf.\ \cite[Remark~2.9]{LMPConeConjecture}). Since $\Pic(X')=\Z L'$, we have that $M=\mu L'$, for some $\mu\in\Q$. We claim that $\mu>0$. In fact, let $\beta\in\oH^{2}(X,\Z)$ be a K\"ahler class, so that $\mu q_{X'}(\beta,L')=\int_{q(C)}\beta>0$; since $L'$ belongs to the border of the positive cone, we must have $q_{X'}(\beta,L')\geq0$, from which the claim follows. 
		
		On the other hand, since $f_*$ is an orientation--preserving Hodge isometry, we must also have $f_*L=\mu' L'$, with $\mu'>0$. It follows that $M = \lambda f_*L$, with $\lambda>0$.
		
		Now let again $\alpha\in\oH^{1,1}(X)$ be a K\"ahler class. On one hand, because of the computation above
		\[
		q_{X'}(f_*(\alpha),M) = f_*\alpha . q(C)\leq 0
		\]
		while, on the other hand,
		\[
		q_{X'}(f_*(\alpha),M) = q_{X}(\alpha,\lambda L) > 0
		\]
		(the last inequality holds because $\lambda L$ is effective by construction and $\alpha$ is a K\"ahler class).
		This contradiction implies that $f$ is an isomorphism.
		
		We conclude that there exists a Hodge isometry that moreover maps a Kähler class to a Kähler class. Since $\mathrm{Aut}_{\mathrm{Hdg}}(\oH^2(X,\Z))=\pm \operatorname{id}$ by \cref{lem:automorphismsParabolicHodgeStructure}, then $(X,\eta)=(X',\eta')$ as claimed.
	\end{proof}

	\proof[Proof of Proposition~\ref{lemma: M ell nef è connesso}]
	First of all, under the assumptions the moduli space $\fM_\ell^0$ is locally path-connected.
	In fact, by the global Torelli theorem (\cite[Theorem~1.1.(4)]{BakkerLehn2022}) its Hausdorffisation is isomorphic to the corresponding period domain, which is a domain of type IV, hence locally path-connected.
	
	The fact that the local path-connectedness lifts from the Hausdorffisation to $\fM_\ell^0$ follows from Lemma~\ref{lemma:cyclic pic is iso}. In fact, the set $\mathcal{W}'=\left\{ (X,\eta)\in\fM_\ell^{0}\mid \dim\, \oH^{1,1}(X,\Z)=1 \right\}$ is dense in $\fM_\ell^0$ and, by the aforementioned Lemma~\ref{lemma:cyclic pic is iso}, isomorphic to its Hausdorffisation; one can then conclude as in the proof of \cite[Corollary~5.11]{Markman:LagrangianFibrations}.
	
	Since by hypothesis $\fM_\ell^{\operatorname{lagr},B}\neq\emptyset$, then by Theorem~\ref{thm:nef=lagr} it is open and dense in $\fM^0_\ell$, therefore it is path-connected.
	
	Finally, the space $\mathcal{W}$ is path-connected since it is the complement of a countable union of real codimension 2 subspaces in a path-connected space (see \cite[Lemma~4.10]{Verbitsky:Torelli}).
	\endproof
	
	\proof[Proof of Theorem~\ref{thm:def as lagr fibr iff existence of suitable pto}]
	If $f_i\colon X_i\to B$ are locally trivial deformation equivalent as lagrangian fibrations, then clearly there exists a locally trivial parallel transport operator sending the class of the fibration to the class of the fibration. Let us then prove the opposite implication.
	
	Let $\eta_2$ be a marking on $X_2$ and let us put $\eta_1=\eta_2\circ \mathsf{P}$. Then by definition $(X_1,\eta_1)$ and $(X_2,\eta_2)$ belong to the same connected component $\fM^0_\Lambda$ of $\fM_\Lambda$. Moreover, there exists an isotropic element $\ell\in\Lambda$ such that $\eta_1(f_1^*h_B)=\ell=\eta_2(f_2^*h_B)$. Again by definition, $(X_1,\eta_1)$ and $(X_2,\eta_2)$ belong to $\fM_{\ell}^0$. More precisely, since both $L_1$ and $L_2$ are classes of lagrangian fibrations, we have that $(X_1,\eta_1)$ and $(X_2,\eta_2)$ belong to $\fM_\ell^{\operatorname{lagr},B}$. 
	
	First of all, we claim that it is enough to prove the case when $\oH^{1,1}(X_i,\Z)=\Z L_i$. In fact, by Proposition~\ref{prop:higher direct images are free}, the infinitesimal universal families $\mathcal{X}_{L_i}\to\Deflt(X_i,L_i)$ are families of lagrangian fibrations and, for a general point $t\in\Deflt(X_i,L_i)$, we have that $\oH^{1,1}(\mathcal{X}_t,\Z)=\Z\mathcal{L}_t$. 
	
	Therefore we can assume that 
	\[ (X_1,\eta_1),(X_2,\eta_2)\in\mathcal{W}=\left\{ (X,\eta)\in\fM_\ell^{\operatorname{lagr},B}\mid \dim\, \oH^{1,1}(X,\Z)=1 \right\}. \]
	
	Now, since $\mathcal{W}$ is path-connected by Proposition~\ref{lemma: M ell nef è connesso}, there exists a path $\gamma\subset\mathcal{W}$ connecting $(X_1,\eta_1)$ to $(X_2,\eta_2)$. Let us choose a finite number of points $p_1,\dots,p_N$ such that $p_1=(X_1,\eta_1)$ and $p_N=(X_2,\eta_2)$. Notice that each $p_k$ corresponds to a marked pair $(Y_k,\eta_k)$ such that there exists a lagrangian fibration $g_k\colon Y_k\to B$. If we put $M_k=c_1(g_k^*h_B)$, then by Proposition~\ref{prop:higher direct images are free} there is a locally trivial family of lagrangian fibrations
	\[
	\xymatrix{
		\mathcal{Y}_k\ar@{->}[rr]\ar@{->}[dr] & & \mathcal{B}_k\ar@{->}[dl] \\
		& \Deflt(Y_k,M_k) &
	}
	\]
	for any $k=1,\dots,N$.
	Notice that, by construction, $\Deflt(Y_k,M_k)\cap\Deflt(Y_{k+1},M_{k+1})\neq\emptyset$. 
	
	Now, for $k=1,\dots,N-1$, let us choose $z_k\in\Deflt(Y_k,M_k)\cap\Deflt(Y_{k+1},M_{k+1})$. On the disjoint union $\coprod_{k=1}^N\Deflt(Y_k,M_k)$, we define the equivalence relation $\sim$ such that for any $x,y\in\coprod_{k=1}^N\Deflt(Y_k,M_k)$, then $x\sim y$ if and only if $x=y$ or there exists an index $k$ such that $x=z_k\in\Deflt(Y_k,M_k)$ and $y=z_k\in\Deflt(Y_{k+1},M_{k+1})$ or vice versa. Let us then define the analytic space $\mathsf{D}=(\coprod_{k=1}^N\Deflt(Y_k,M_k))/\sim$.
	
	Similarly, let us define the spaces:
	\begin{itemize}
		\item $\mathcal{Y}$ by gluing, for every $k=1\dots,N-1$, the spaces $\mathcal{Y}_k$ and $\mathcal{Y}_{k+1}$ at the fibres $(\mathcal{Y}_k)_{z_k}\cong(\mathcal{Y}_{k+1})_{z_k}$;
		\item $\mathcal{B}$ by gluing, for every $k=1\dots,N-1$, the spaces $\mathcal{B}_k$ and $\mathcal{B}_{k+1}$ at the fibres $(\mathcal{B}_k)_{z_k}\cong(\mathcal{B}_{k+1})_{z_k}$.
	\end{itemize}
	In this way we get a locally trivial family of lagrangian fibrations 
	\[
	\xymatrix{
		\mathcal{Y}\ar@{->}[rr]^{\tilde{f}}\ar@{->}[dr] & & \mathcal{B}\ar@{->}[dl] \\
		& \mathsf{D} &
	}
	\]
	such that $\mathsf{D}$ is connected and there exist two points $d_1,d_2\in\mathsf{D}$ such that $\tilde{f}_{d_i}\colon\mathcal{Y}_{d_i}\to\mathcal{B}_{d_i}$ coincides with $f_i\colon X_i\to B$.
	This concludes the proof.
	\endproof
	
	\begin{rem}\label{rmk:term+Q-fact}
		The hypotheses of terminality and $\Q$-factoriality in Theorem~\ref{thm:def as lagr fibr iff existence of suitable pto} are needed in our proof of Proposition~\ref{lemma: M ell nef è connesso}. Nevertheless, we expect versions of Theorem~\ref{thm:def as lagr fibr iff existence of suitable pto} without these assumptions.
	\end{rem}

	%%%%%%%%%%%%%%%%%%%%%%%%%%%%%%
	%%%%%%%%%%%%%%%%%%%%%%%%%%%%%%
	\section{Moduli spaces of sheaves on K3 surfaces}\label{section:moduli spaces of sheaves on K3}
	In this section we recall some facts about moduli spaces of sheaves on K3 surfaces and their relations with primitive symplectic varieties.
	
	%%%%%%%%%%%%%%%%%%%%%%%%%%%%%%%%%%%%%
	%%%%%%%%%%%%%%%%%%%%%%%%%%%%%%%%%%%%%%%
	\subsection{Generalities}\label{section:generalities}
	Let $S$ be a projective K3 surface. Recall that the Mukai lattice of $S$ is 
	\[ \widetilde{\oH}(S,\Z):=\oH^0(S,\Z)\oplus\oH^2(S,\Z)\oplus\oH^4(S,\Z),\qquad (r,c,s)^2=c^2-2rs, \]
	and it comes with a weight--two Hodge structure such that $\widetilde{\oH}(S,\C)^{2,0}=\oH^{2,0}(S,\C)$.
	
	A vector $v=(r,c,s)\in\widetilde{\oH}(S,\Z)$ is a \emph{Mukai vector} if $r\geq0$ and $c\in\oH^{1,1}(S,\Z)$, and if $r=0$, then either $c$ is strictly effective or $c=0$ and $s>0$. These properties ensure that there exists a coherent sheaf $F$ on $S$ such that $v(F):=\ch(F)\sqrt{\td_S}=v$.
	
	Once an ample class $H$ on $S$ is fixed, we will consider the moduli space $M_v(S,H)$ of Gieseker--Maruyama $H$-semistable sheaves $F$ on $S$ such that $v(F)=v$.
	
	In order to have a well-behaved moduli space, we ask that the ample class $H$ is chosen general with respect to $v$ (see \cite[Definition~2.8]{PR:SingularVarieties}). We will not recall here the definition of generality, but we will only list the properties we will use:
	\begin{itemize}
		\item if $\Pic(S)=\Z H$, then $H$ is always general with respect to $v$ (cf.\ \cite[Lemma~2.9]{PR:SingularVarieties});
		\item being general with respect to $v$ is a Zariski open condition in families (cf.\ \cite[Proposition~2.14]{PR:SingularVarieties});
		\item let $S$ and $S'$ be two projective K3 surfaces, $v=mw$ a Mukai vector on $S$ and $v'=m w'$ a Mukai vector on $S'$, such that $w^2=(w')^2$; if $H$ is general with respect to $v$ and $H'$ is general with respect to $v'$, then the moduli space $M_v(S,H)$ is locally trivial deformation equivalent to the moduli space $M_{v'}(S',H')$ (see \cite[Theorem~1.7]{PR:SingularVarieties}).
	\end{itemize}
	
	If the Mukai vector $v$ is primitive and the ample class $H$ is general with respect to $v$, then the moduli space $M_v(S,H)$ is an irreducible holomorphic symplectic manifold deformation equivalent to the Hilbert scheme $\Hilb^{\frac{v^2+2}{2}}(S)$ (see \cite{OGrady:WeightTwo,Yoshioka:ModuliAbelian}). If $v$ is not primitive, we write $v=mw$, where $m>0$ and $w$ is a primitive Mukai vector. 
	
	The following summarises the results we will need later.
	
	\begin{thm}[\cite{PR:SingularVarieties,PR:vperp,OPR}]\label{thm:PR}
		Let $S$ be a projective K3 surface, $v$ a Mukai vector and $H$ an ample class that is general with respect to $v$. Write $v=mw$, with $m>0$ and $w$ a primitive Mukai vector such that $w^2 = 2k >0$. Then:
		\begin{enumerate}
			\item the moduli space $M_v(S,H)$ is an irreducible symplectic variety of dimension $v^2+2$ (\cite[Theorem~1.10]{PR:SingularVarieties});
			\item the locally trivial deformation equivalence class of $M_v(S,H)$ only depends on the pair $(m,k)$ (\cite[Theorem~1.7]{PR:SingularVarieties});
			\item there exists an Hodge isometry
			\[ \lambda\colon v^\perp\longrightarrow\oH^2(M_v(S,H),\Z), \]
			where $v^\perp$ inherits the lattice and Hodge structures from the Mukai lattice $\widetilde{\oH}(S,\Z)$ and $\oH^2(M_v(S,H),\Z)$ is endowed with the BBF lattice structure (\cite[Theorem~1.6]{PR:vperp});
			\item the Fujiki constant of $M_v(S,H)$ is
			\[ C_v=\frac{(2n)!}{n! 2^n} \]
			where $2n=\dim M_v(S,H)$ (\cite[Theorem~1.7]{PR:vperp});
			\item the locally trivial monodromy group does not depend on $m$ and it is equal to
			\[ \Mon(M_v(S,H))=\mathsf{W}(v^\perp) \]
			where $\mathsf{W}(v^\perp)$ is the group of orientation--preserving isometries of $v^\perp$ acting as $\pm\operatorname{id}$ on the discriminant group (\cite[Theorem~A.2]{OPR}).
		\end{enumerate}
	\end{thm}
	
	\begin{defn}\label{defn: M(m,k)}
		A primitive symplectic variety locally trivial deformation equivalent to a moduli space $M_v(S,H)$, with $v=mw$ and $w^2=2(k-1)$, as in \cref{thm:PR} will be said to be of type $\operatorname{K3}^{[k]}_m$.
	\end{defn}
	
	%%%%%%%%%%%%%%%%%%%%%%%%%%%%%%%%%%%%%%
	\subsection{Beauville--Mukai systems}\label{section:BM system}
	Let $S$ be a projective K3 surface such that $\Pic(S)=\Z H$, and suppose that $H^2=2d$. Let us fix a Mukai vector $v=(0,mH,ms)$. Notice that, by \cite[Lemma~2.9]{PR:SingularVarieties}, the ample class $H$ is general with respect to $v$ and hence the moduli space $M_v(S,H)$ is an irreducible symplectic variety. When $m=1$ the moduli space $M_{(0,H,s)}(S,H)$ is smooth and deformation equivalent to the Hilbert scheme $\Hilb^{d+1}(S)$.
	
	Mapping every sheaf $F\in M_v(S,H)$ to its Fitting support gives a morphism (see \cite[Section~1.4]{Moz})
	\[ p_v\colon M_v(S,H)\longrightarrow |mH|\cong\P^{m^2d+1} \]
	such that the fibre over a smooth and irreducible curve $C\in|mH|$ is the Picard variety $\Pic^{\delta}(C)$, where $\delta=m(md+s)$. This fibration is equidimensional and it is a compactification of the relative Picard variety over the locus of smooth curves in $|mH|$. We refer to this lagrangian fibration as a \emph{Beauville--Mukai system}.
	
	In particular $M_v(S,H)$ parametrises torsion sheaves of rank $1$ on their support. A general point of $M_v(S,H)$ is of the form $i_*L$, where $i\colon C\to S$ is the closed embedding of a smooth curve $C\in|mH|$ and $L$ is a line bundle on $C$ such that $\chi(L)=ms$, i.e.\ $L$ has degree $\delta$.
	
	The morphism $p_v$ is a lagrangian fibration (see \cref{section:LagrangianFibrations}).
	
	By Theorem~\ref{thm:PR}, there is an isometry
	\[ \lambda_v\colon (v^\perp)^{1,1}\longrightarrow\Pic(M_v(S,H)). \]
	Let us distinguish two cases:
	\begin{description}
		\item[$s=0$] then $(v^\perp)^{1,1}$ is a unimodular hyperbolic plane generated by the two isotropic classes $a=(-1,0,0)$ and $b=(0,0,1)$, i.e.
		\[ (v^\perp)^{1,1}=\langle a,b\rangle=\left(
		\begin{array}{cc} 0 & 1 \\ 1 & 0 \end{array} \right);
		\]
		\item[$s\neq0$] then $(v^\perp)^{1,1}$ is a non-unimodular hyperbolic plane generated by the two isotropic classes $a=(\frac{2d}{\mu},\frac{s}{\mu}H,0)$ and $b=(0,0,1)$, where $\mu=\gcd(d,s)$, i.e.
		\[ (v^\perp)^{1,1}=\langle a,b\rangle=\left(
		\begin{array}{cc} \frac{2ds^2}{\mu^2} & -\frac{2d}{\mu} \\ -\frac{2d}{\mu} & 0 \end{array} \right).
		\]
	\end{description}
	
	When $m=1$, i.e.\ $v$ is primitive, it is known that the class $\lambda_v(b)$ represents the class $p_v^*\sO(1)$. This can be checked by hand as in the proof of \cite[Lemma~6.5.(iii)]{Wieneck2016}; alternatively, one can look at \cite[Lemma~11.3]{BM:MMP}. 
	
	The same is true for any $m\geq2$.
	
	\begin{lem}\label{lemma:b=b}
		If $p_v\colon M_v(S,H)\to|mH|$ is as above, then $p_v^*\sO(1)=\lambda_v(b)$.
	\end{lem}
	Notice in particular that $p_v^*\sO(1)$ is primitive.
	\begin{proof}
		Let us write $v=mw$, with $m\geq2$. Then we have a commutative diagram
		\begin{equation}\label{eqn:w in v}
			\xymatrix{
				M_v(S,H) \ar[d]_{p_v} & M_w(S,H) \ar[l]_{i_{m,w}} \ar[d]^{p_w} \\
				|mH| & |H| \ar[l]^{\nu_m}
			}
		\end{equation}
		where $i_{m,w}$ is the closed embedding of $M_w(S,H)$ as the most singular stratum of $M_v(S,H)$ and $\nu_m$ is the Veronese embedding. 
		More precisely, we have two cases: if $H^2=2$, in which case $H$ is not very ample, then $\nu_m$ is the Veronese embedding followed by a linear embedding; if $H^2>2$, in which case $H$ is very ample, then the composition of $\nu_m$ and a linear embedding is the Veronese embedding. In both cases, it follows that $\nu_m^*\sO(1)=\sO(m)$.
		
		Let us write $p_v^*\sO(1)=\lambda_v(c)$, for some class $c\in v^\perp$. 
		Then, by \cite[Proposition~1.28]{OPR} we have that
		\[ i_{m,w}^*(p_v^*\sO(1))=m\lambda_w(c). \] 
		
		By the commutativity of the square (\ref{eqn:w in v}) we eventually get
		\[ m\lambda_w(c)=i_{m,w}^*(p_v^*\sO(1))=p_w^*(\nu_m^*\sO(1))=p_w^*\sO(m)=m\lambda_w(b) \]
		from which the claim follows.
	\end{proof}
	
	\begin{rem}
		The same proof shows that \cref{lemma:b=b} holds more generally for projective K3 surfaces of any Picard rank and any $v$-general polarisation.
	\end{rem}
	
	The following result is \cite[Theorem~1.7]{PR:SingularVarieties}, we state it here for the sake of completeness.
	
	\begin{lem}\label{lem:any MS is def of BM}
		Let $X$ be a variety of type $\operatorname{K3}^{[k]}_m$. Then $X$ is locally trivial deformation equivalent to a Beauville--Mukai system $p_v\colon M_v(S,H)\longrightarrow |mH|$. 
	\end{lem}
	\begin{proof}
		By definition, since $X$ is of type $\operatorname{K3}^{[k]}_m$, there exists a projective K3 surface $T$, a primitive Mukai vector $w_T\in\widetilde\oH(T,\Z)$ with $v_T^2=2k$ and an ample class $H_T\in\Pic(T)$ that is $v_T$-general (here $v_T=mw_T$), such that $X$ is locally trivial deformation equivalent to $M_{v_T}(T,H_T)$. 
		
		If $v_T$ is of the form $(0,m\xi,mt)$, then $M_{v_T}(T,H_T)$ is a Beauville--Mukai system and we are done. Otherwise, let $S$ be a projective K3 surface, $w_S=(0,\ell,s)$ a Mukai vector with $w_S^2=2k$, and $H_S\in\Pic(S)$ an ample line bundle that is $v_S$-general (again, here $v_S=mw_S$). Then by \cite[Theorem~1.7]{PR:SingularVarieties} the two moduli spaces $M_{v_T}(T,H_T)$ and $M_{v_S}(S,H_S)$ are locally trivial deformation equivalent, and $M_{v_S}(S,H_S)$ is a Beauville--Mukai system by construction. This concludes the proof.
	\end{proof}
	
	%%%%%%%%%%%%%%%%%%%%%%%%%%%%%%%%%%%%%%%%%%%%%%
	%%%%%%%%%%%%%%%%%%%%%%%%%%%%%%%%%%%%%%%%%%%%%%
	\section{The Huybrechts--Riemann--Roch polynomial of moduli spaces of sheaves on K3 surfaces}\label{section:HRR for MSK3}
	
	Recall from Definition~\ref{defn: M(m,k)} that a variety of type $\operatorname{K3}^{[k]}_m$ is a primitive symplectic variety that is locally trivial deformation equivalent to a moduli space of sheaves on a K3 surfaces as in Section~\ref{section:moduli spaces of sheaves on K3}.
	
	The aim of this section is to prove the following statement.
	
	\begin{thm}\label{thm: RR of K3n type}
		Let $X$ be of type $\operatorname{K3}^{[k]}_m$, then the Huybrechts--Riemann--Roch polynomial of $X$ is of $K3^{[n]}$-type, where $n = (k-1)m^2+1$.
	\end{thm}
	The result is known when $m=1$ (i.e.\ when the moduli space is smooth and of $\operatorname{K3}^{[n]}$-type), and when $(m,k)=(2,1)$. The latter is not explicitly stated in the literature, but it essentially follows from \cite{Rios2020}.
	
	Up to locally trivial deformation, the proof will be reduced to considering a particular Beauville--Mukai system. In particular, we will need to perform some computations about the theta divisor, which we perform in Section~\ref{section:theta} below.
	
	%%%%%%%%%%%%%%%%%%%%%%%%%%%%%%%%%%%%%%%%%%%%%%
	\subsection{An effective relative theta divisor}\label{section:theta}
	
	Let $S$ be a projective K3 surface and $H$ an ample class such that $H^2=2d$. We further assume that $\Pic(S)=\Z H$. Let us consider the Mukai vector 
	\[ v=(0,mH,0) \]
	so that the moduli space $M_v(S,H)$ is a primitive symplectic variety and we are in the setting of Section~\ref{section:BM system}. In particular, there is a lagrangian fibration
	\[
	\pi\colon M_v(S,H)\longrightarrow |mH|\cong\P^{m^2d+1}
	\]
	whose smooth fibres are Jacobian varieties of degree $m^2d$ of smooth curves of genus $m^2d+1$.
	
	In this case we have that
	\[ \Pic(M_v(S,H))\cong\langle a,b\rangle=\left(
	\begin{array}{cc} 0 & 1 \\ 1 & 0 \end{array} \right),
	\]
	where $a=(-1,0,0)$ and $b=(0,0,1)$ (see \cref{section:BM system}).
	
	Let us now consider the following subvariety
	\[ D_v:=\{F\in M_v(S,H)\mid h^0(F)\geq1\}. \]
	By Brill--Noether theory we can see that $D_v$ is non-empty and has codimension one, i.e.\ 
	\begin{align*} 
		\dim D_v & = m^2d+1+\rho(g,d,r) \\
		& =m^2d+1+\rho(m^2d+1,m^2d,0) \\
		& = 2m^2d+1 = \dim M_v-1.
	\end{align*}
	Alternatively, we can identify $D_v$ with the image of the relative Abel--Jacobi morphism
	\[ \operatorname{Sym}_{|mH|}^{m^2d}\mathcal{C}\dashrightarrow M_v(S,H), \]
	where $\mathcal{C}$ is the family of curves parametrised by $|mH|$, from which it also follows that it is non-empty and of codimension one. Notice that the restriction of $D_v$ to a general fibre of $\pi$ is an effective theta divisor. 
	
	A priori $D_v$ is only a Weil divisor, but by \cite[Theorem~A]{KaledinLehnSorger} (when $m>2$ or $m=2$ and $d>1$) and \cite[Theorem~1.1]{PR:Factoriality} (when $m=2$ and $d=1$), we know that $M_v(S,H)$ is locally factorial, so that $D_v$ is in fact a Cartier divisor.
	
	\begin{defn}\label{defn:theta}
		We call $D_v$ the effective relative theta divisor, and we denote by $\Theta_v$ the class of $D_v$ in $\oH^{1,1}(M_v(S,H),\Z)$.
	\end{defn}
	
	The main result of this section is the following.
	\begin{prop}\label{prop:pex}
		$D_v$ is a prime exceptional divisor and its class $\Theta_v$ satisfies $\Theta_v^2=-2$ and $\divv(\Theta_v)=1$. 
	\end{prop}

	\begin{rem}
		The fact that $D_v$ is prime exceptional was already remarked in \cite[Section~4.2]{LMP}. Our improvement with respect to their result is that we explicitly compute its degree and divisibility.
	\end{rem}
	
	We will dedicate the rest of the section to proving the proposition. Our first claim is the following.
	\begin{lem}\label{lemma:uniruled}
		There exists a rational morphism $q\colon D_v\dashrightarrow S^{[m^2d]}$ whose general fibre is isomorphic to $\P^{1}$.
	\end{lem}
	\begin{proof}
		Let $F=i_*L$ be a general sheaf in $D_v$. In particular $L$ is a line bundle of degree $m^2d$ on a smooth curve $C\in|mH|$. Since $\chi(L)=0$ and $h^0(L)\geq1$, by Serre duality it follows that there exists a non-trivial morphism $\epsilon\in\Hom(L,\omega_C)$, where $\omega_C$ is the canonical line bundle of $C$. In particular there exists a short exact sequence 
		\[ 0\to L\stackrel{\epsilon}{\to} \omega_C\to \sO_\xi\to0 \]
		where $\xi$ is a $0$-dimensional subscheme of $C$ of length $\ell$. Moreover, since $\chi(L)=0$, we get that $\ell=\chi(\sO_\xi)=\chi(\omega_C)=m^2d$.
		
		Since $D_v$ is a relative theta divisor, there exists an open subset in $D_v$ of sheaves $F$ such that $h^0(F)=1$. In particular, since the above construction works in families, there is a rational morphism $q\colon D_v\dashrightarrow S^{[m^2d]}$ as claimed.
		
		Let us now describe the general fibre of $q$.
		If $\xi\in S^{[m^2d]}$, then there is a commutative diagram
		\[ 
		\xymatrix{
			& &  0\ar@{->}[d]  &  0\ar@{->}[d] &  \\
			0\ar@{->}[r] & \sO_S\ar@{->}[r]\ar@{=}[d] & I_\xi(mH)\ar@{->}[r]\ar@{->}[d] &  F\ar[r]\ar[d] & 0 \\
			0\ar@{->}[r] & \sO_S\ar@{->}[r] & \sO_S(mH)\ar@{->}[r]\ar@{->}[d] & \omega_C\ar@{->}[r]\ar@{->}[d] & 0 \\
			& & \sO_\xi\ar@{=}[r]\ar[d] & \sO_\xi\ar[d] & \\
			& &  0  &  0 &  .
		}
		\]
		The fibre $q^{-1}(\xi)$ consists of those sheaves $F\in M_v(S,H)$ such that there exists a short exact sequence 
		\[ 0\to\sO_S\to I_\xi(mH)\to F\to 0, \]
		which corresponds to the choice of a section $s\in\oH^0(I_\xi(mH))$. A simple computation shows that $h^0(I_\xi(mH))\geq2$. If $\xi$ is very general, then $\xi$ corresponds to $m^2d$ points in general position, so that $h^0(I_\xi(mH))=2$. If we consider the pencil of curves $\P\oH^0(I_\xi(mH))$ passing through the general points in the support of a very general $\xi$, then there is a smooth morphism 
		\[ \P\oH^0(I_\xi(mH))\longrightarrow D_v \]
		explicitly given by $[T]\mapsto j_*\sO_T(\xi)$, where $j\colon T\to S$ is the closed embedding. In particular, the general fibre of $q$ is a smooth rational curve as claimed.
	\end{proof}
	
	By Lemma~\ref{lemma:uniruled} we get that $D_v$ is uniruled. Moreover, the general rational curve ruling $D_v$ is smooth, so that $D_v$ is prime exceptional (cf.\ \cref{section:pex}). Notice also that the class $\gamma$ of the general fibre of $q\colon D_v\dashrightarrow S^{[m^2d]}$ is then proportional to the dual class $\Theta_v^\vee$.
	\medskip
	
	We want to give a modular interpretation of the general fibre of $q\colon D_v\dashrightarrow S^{[m^2d]}$, in order to be able to compute its class in $(v^\perp)^{1,1}\otimes\Q$. Let then $\xi\in S^{[m^2d]}$ be a general point and let us consider $C=\P\oH^0(I_\xi(mH))$. We denote by $\pi_S$ and $\pi_C$ the projections from $S\times C$ to $S$ and $C$, respectively. There exists an injective morphism of sheaves (see \cite[Section~2.2, Appendix]{Perego:2-factoriality}),
	\begin{equation}\label{eqn:cF} 
		\pi_S^*\sO_S\otimes\pi_C^*\sO_C(-1)\hookrightarrow\pi_S^*I_\xi(mH), 
	\end{equation}
	and we denote by $\cF$ its cokernel. Then $\cF$ is a sheaf on $S\times C$, flat over $C$, that parametrises semistable sheaves in $M_v$. The classifying morphism
	\[ \phi_{C,\cF}\colon C\to M_v(S,H) \]
	is not constant, and it defines a line in $M_v(S,H)$ that we denote by $C$ again. 
	
	Let $\gamma\in\oH_2(M_v(S,H),\Z)$ be the homology class of $C$.
	Using the isomorphism (\ref{eqn:iso}) we can write $\gamma$ as a rational linear combination of the classes $a=(-1,0,0)$ and $b=(0,0,1)$. 
	
	\begin{lem}\label{lemma:intersections}
		With notation as above,
		\[ 
		\gamma.\lambda(a)=-1\qquad\mbox{ and }\qquad \gamma.\lambda(b)=1.
		\]
	\end{lem}
	\begin{proof}
		The proof is the same as in \cite[Lemma~4.11]{Onorati:Monodromy}; we quickly sketch the main points. 
		First of all, 
		\[ \gamma.\lambda(a)=\pi_{C*}\left[\ch(\cF)\pi_S^*(a^\vee\sqrt{\td_S})\right]_3, \] 
		where, by a direct computation using (\ref{eqn:cF}),
		\[ \ch(\cF)=(0,m[H\times C]+[S\times\operatorname{pt}], mt[\operatorname{pt}\times C],0) \]
		and
		\[ a^\vee\sqrt{\td_S}=(-1,0,-1). \]
		Therefore,
		\[ \gamma.\lambda(a)=-1. \]
		Similarly, since $b^\vee\sqrt{\td_S}=(0,0,1)$, we get
		\[ \gamma.\lambda(b)=1. \]
	\end{proof}
	
	Therefore we get
	\begin{equation}\label{eqn:L} 
		\gamma=\lambda(a)-\lambda(b) 
	\end{equation}
	so that $\gamma$ is integral, i.e.\ it lies in $\oH^2(M_{v}(S,H),\Z)$.

	To conclude the proof of Proposition~\ref{prop:pex}, we will now show that $\Theta_v=\gamma$. In fact, we already know that they are proportional, so it will be enough to find a class $\alpha\in\oH^2(M_v(S,H),\Z)$ such that 
	\[ q_v(\Theta_v,\alpha)=q_v(\gamma,\alpha). \]
	We choose $\alpha=b_v$, where $b_v=p_v^*\sO(1)$ is the class of the fibration. 
	Let us notice that, by construction, $\gamma$ is the class of a curve $C$ that is a section for the fibration $p_v\colon M_v(S,H)\to|mH|$. Therefore, since $\gamma$ is integral, we have $q_v(\gamma,b_v)=1$.
	
	On the other hand, if we put $\dim M_v(S,H)=2n$, then by Proposition~\ref{prop:Fujiki} and Theorem~\ref{thm:PR} we have 
	\[ \int_{M_v(S,H)}(\Theta_v+tb_v)^{2n}=\frac{(2n)!}{n!2^n} q_{v}(\Theta_v+tb_v)^n. \]
	Equalising the coefficients of $t^n$ on both sides, we get
	\[ \frac{(2n)!}{n!} q_v(\Theta_v,b_v)^n=\binom{2n}{n}\int_{M_v(S,H)} \Theta_v^n b_v^n=\binom{2n}{n}\int_{\Pic^{m^2d}(C)} \Theta_v^n=\frac{(2n)!}{n!}, \]
	where $\Pic^{m^2d}(C)$ is a general fibre of $p_v\colon M_v(S,H)\to\P^{m^2d+1}$, and where we used that $(\Theta_v)|_{\Pic^{m^2d}(C)}$ is the class of the theta divisor on $C$ to get that $\int_{\Pic^{m^2d}(C)} \Theta^n=n!$. 
	
	It follows that $q_v(\Theta_v,b_v)=1$, and hence 
	\begin{equation}
		\Theta_v=\gamma=\lambda(a)-\lambda(b).
	\end{equation}
	\endproof

	\begin{rem}
		Let us write $v=mw$, where $w$ is a primitive Mukai vector and $m>1$. Let $H$ be an ample line bundle that is general with respect to both $v$ and $w$. The effective relative theta divisors $\Theta_v$ and $\Theta_w$ both have degree $-2$ and divisibility $1$: in fact the statement of Proposition~\ref{prop:pex} holds for every $m\geq1$. Let $i_{m,w}\colon  M_w(S,H)\to M_v(S,H)$ be the closed embedding as the most singular locus. Then by \cite[Proposition~1.28]{OPR} and the definitions of $D_v$ and $D_w$ it follows that  
		\[ i^*(\Theta_v)=m\Theta_w. \]  
	\end{rem}
	
	%%%%%%%%%%%%%%%%%%%%%%%%%%%%%%%%%%%%%%%%
	\subsection{Proof of Theorem \ref{thm: RR of K3n type}}
	
	Let $S$, $H$ and $v$ be as in \cref{section:theta}.
	Then $M_v(S,H)$ is factorial, the Picard group of $X$ has rank $2$ and there is a natural lagrangian fibration $p_v\colon M_v(S,H)\to|mH|$. For simplicity, let us put $X=M_v(S,H)$ and let us denote by $X_t$ the fibre of $p_v$ over $t\in|mH|$.
	
	Let $D_v$ be the relative theta divisor and $\Theta_v$ its class (see \cref{defn:theta}). Put  $b_v := p_v^*\O(1)$; then by Proposition~\ref{prop:pex} we get $q_v(\Theta_v) = -2$ and $q_v(\Theta_v,b_v) = 1$. Moreover, the restriction of $\Theta_v$ to every smooth fibre $X_t$ defines a principal polarisation, since it is a theta divisor on the Jacobian of a smooth curve.
	
	\begin{lem}
		The class $\Theta_v$ is $p_v$-ample and $p_{v*}\O_X(D_v) = \O_{\P^n}$.
	\end{lem}
	\begin{proof}
		Since $\Theta_v$ is effective and $b_v$ is nef, we have that the divisor $\Theta_v + kb_v$ will be ample for some $k>0$. 
		
		By \cite[Theorem~3]{Rios2020} the sheaf $p_{v*}\O_X(D_v)$ is a line bundle on $\P^n$. Since $q_v(\Theta_v)<0$ we get $h^0(X,\O_X(D_v)) =1$ and, using that $\oH^0(X,\O_X(D_v)) = \oH^0(\P^n,p_{v*}\O_X(D_v))$, this implies that $p_{v*}\O_X(D_v) = \O_{\P^n}$.
	\end{proof}
	
	Recall that by \cref{exmp: RR for known HK} there are exactly two types of Huybrechts--Riemann--Roch polynomials for the known smooth hyper-Kähler manifolds. The main result of this section completes the computation of the Huybrechts--Riemann--Roch polynomials for all, not necessarily smooth, moduli spaces of sheaves on a K3 surface.
	
	\proof[Proof of Theorem~\ref{thm: RR of K3n type}]
	By definition, $X$ is locally trivial deformation equivalent to a moduli space $M_v(S,H)$ as in \cref{section:moduli spaces of sheaves on K3}. By Lemma~\ref{lem:any MS is def of BM}, it is then locally trivial deformation equivalent to a moduli space with Mukai vector of the form $(0,mH,0)$. Since the Huybrechts--Riemann--Roch polynomial is invariant under locally trivial deformations (see \cref{RRpolynomial}), it is enough to prove the claim in this case.
	
	Let then $D_v$ be the effective relative theta divisor, and $b_v$ the class of the fibration.
	The higher direct images of $p_{v*}\O_X(D_v+kb_v)$ vanish by \cite[Theorem 3]{Rios2020}, therefore
	\begin{equation}\label{eq:binomialRR}
		\chi(X,\Theta_v+kb_v) = \chi(\P^n, \O(k)) = \binom{k+n}{n}.
	\end{equation}
	By a direct computation we have $q_v(\Theta_v+kb_v) =  q_v(\Theta_v) +2k$. Substituting $t = q_v(\Theta_v + kb_v)$ in \eqref{eq:binomialRR}, we eventually get
	\[
	\RR_X(t) := \binom{\frac{t-q_v(\Theta_v)}{2} + n}{n} = \binom{\frac{t}{2} + n + 1}{n}.
	\]
	This ends the proof.
	\endproof
	
	%%%%%%%%%%%%%%%%%%%%%%%%%%%%%%%%%%%%%%%%%%%%%%%
	%%%%%%%%%%%%%%%%%%%%%%%%%%%%%%%%%%%%%%%%%%%%%%%
	\section{Lagrangian fibrations of moduli spaces of sheaves on K3 surfaces}\label{section:SYZ for M(m,k)}
	
	The goal of this section is to prove the following result.
	
	\begin{thm}\label{thm:SYZ for M(m,k)}
		Let $X$ be an irreducible symplectic variety of type $\operatorname{K3}^{[k]}_m$ . Assume that $m>1$ and that if $m=2$ then $k>2$. If $L\in\Pic(X)$ is a line bundle that is primitive, nef and isotropic, then $L$ induces a lagrangian fibration over $\P$ (in the sense of Definition~\ref{defn:L induces a lagr fibr}).
	\end{thm}
	
	We start by extending a result due to Markman \cite[Sections~2 and 3]{Markman:LagrangianFibrations} (see also \cite[Section~6]{Wieneck2016}).
	
	\begin{prop}\label{prop:quasi SYZ}
		Let $X$ be an irreducible symplectic variety of type $\operatorname{K3}^{[k]}_m$ as in Theorem~\ref{thm:SYZ for M(m,k)}, and let $L\in\operatorname{NS}(X)$ be a primitive and isotropic class of divisibility $d$. Then there exists a Beauville--Mukai system $p_v\colon M_v(S,H)\to|mH|$ as in \cref{section:BM system} and a locally trivial parallel transport operator
		\[ \mathsf{P}\colon \oH^2(X,\Z)\longrightarrow \oH^2(M_v(S,H),\Z) \]
		such that $\mathsf{P}(L)=p_v^*\sO(1)$.
	\end{prop}
	
	\begin{proof}    
		Let $Y\subset X$ be the most singular locus of $X$. Then $Y$ is an irreducible holomorphic symplectic manifold of type $\operatorname{K3}^{[k]}$.    
		Let us denote by 
		\[ i_Y\colon Y\longrightarrow X \]
		the closed embedding. Then the pullback in cohomology 
		\[ i_Y^*\colon\oH^2(X,\Z)\longrightarrow\oH^2(Y,\Z) \]
		is $m$--times a Hodge isometry. In fact the stratification by singular loci of $X$ behaves well in locally trivial families, so that the embedding $i_Y$ fits in a local system and it will be enough to prove the claim for a preferred member of a family. By definition we can choose a family having a moduli space $M_v(S,H)$ as a member, so that the claim follows from \cite[Proposition~1.28]{OPR}.
		
		It follows that there exists an isotropic class $L_0\in\operatorname{NS}(Y)$ of divisibility $d$ such that $i_Y^*(L)=m\,L_0$.
		
		By the results in \cite[Sections~2 and 3]{Markman:LagrangianFibrations} and \cite[Section~6]{Wieneck2016}, it follows that there exists a projective K3 surface of Picard rank $1$, a Beauville--Mukai system $p_w\colon M_w(S,H)\to|H|$, with $w=(0,H,s)$ a primitive Mukai vector, and a parallel transport operator 
		\[ \mathsf{P}_0\colon\oH^2(Y,\Z)\longrightarrow\oH^2(M_w(S,H),\Z) \]
		such that $\mathsf{P}_0(L_0)=p_w^*\sO(1)$.
		
		Put $v=mw$ and consider the Beauville--Mukai system $p_v\colon M_v(S,H)\to|mH|$. We claim that there exists a locally trivial parallel transport operator $\mathsf{P}\colon \oH^2(X,\Z)\longrightarrow \oH^2(M_v(S,H),\Z)$ such that the following diagram 
		\begin{equation}\label{eqn:boh} 
			\xymatrix{
				\oH^2(X,\Z) \ar[d]_{i_Y^*} \ar[r]^-{\mathsf{P}} & \oH^2(M_v(S,H),\Z) \ar[d]^{i_{m,w}^*}  \\
				\oH^2(Y,\Z) \ar[r]^-{\mathsf{P}_0} & \oH^2(M_w(S,H),\Z)
			} 
		\end{equation}
		is commutative.
		
		Assuming the claim, we have
		\[ i_{m,w}^*(\mathsf{P}(L))=\mathsf{P}_0(i_Y^*(L))=\mathsf{P}_0(m\, L_0)=p_w^*\sO(m). \]
		By \cref{lemma:b=b} it then follows that $\mathsf{P}(L)=p_v^*\sO(1)$, thus concluding the proof.
		
		To prove the claim, let us take a locally trivial parallel transport operator 
		\[ \mathsf{P}'\colon\oH^2(M_v(S,H),\Z)\to\oH^2(X,\Z) \]
		and let $\mathsf{P}_0'\colon\oH^2(M_w(S,H),\Z)\to\oH^2(Y,\Z)$ be the induced parallel transport operator between the most singular loci. In particular we have $\mathsf{P}'=(i_Y^*)^{-1} \circ \mathsf{P}_0' \circ i_{m,w}^*$. Then by definition $g_0=\mathsf{P}_0'\circ\mathsf{P}_0\in\operatorname{Mon}^2(Y)$.
		By \cite[Theorem~B.1]{OPR} we have\footnote{Notice that both $i_Y^*$ and $i_{m,w}^*$ are $m$--times an isometry.} 
		\begin{align*} 
			\mathsf{P}'\circ \left((i_{m,w}^*)^{-1} \circ \mathsf{P}_0\circ i_Y^*\right) & = \left( (i_Y^*)^{-1} \circ \mathsf{P}_0' \circ i_{m,w}^* \right)\circ \left((i_{m,w}^*)^{-1} \circ \mathsf{P}_0\circ i_Y^*\right) \\
			& = (i_Y^*)^{-1}\circ g_0\circ i_Y^*\in\Mon(X). 
		\end{align*}
		Therefore 
		\[ \mathsf{P}:=(i_{m,w}^*)^{-1} \circ \mathsf{P}_0\circ i_Y^*\colon\oH^2(X,\Z)\longrightarrow\oH^2(M_v(S,H),\Z) \]
		is the desired locally trivial parallel transport operator.
	\end{proof}
	
	\proof[Proof of Theorem~\ref{thm:SYZ for M(m,k)}]
	By Proposition~\ref{prop:quasi SYZ}, there exists a locally trivial parallel transport operator 
	\[ \mathsf{P}\colon \oH^2(X,\Z)\longrightarrow \oH^2(M_v(S,H),\Z) \]
	such that $\mathsf{P}(L)=p_v^*\sO(1)$. Since $L$ is nef by hypothesis, the claim follows at once from Theorem~\ref{thm:main result on lagr}.
	\endproof
	
	\begin{rem}
		By Theorem~\ref{thm:def as lagr fibr iff existence of suitable pto}, the locally trivial deformation class of a pair $(X,L)$ as in Theorem~\ref{thm:main result on lagr} is determined by the monodromy-orbit of $L$, which only depends on the divisibility of $L$ and the dimension of $X$, see \cite[Lemma~2.5]{Markman:LagrangianFibrations}.
	\end{rem}

	\bibliography{bib}
	\bibliographystyle{alpha}
	
\end{document}